\documentclass{article}
\usepackage{amsmath,amsfonts,amsthm,amssymb,amscd,color,xcolor}
\usepackage{graphicx,amscd,mathrsfs,eufrak,wrapfig,mathrsfs,lipsum}
\usepackage{microtype}
\usepackage{float}
\usepackage{tikz}
\usepackage{multicol}
\usepackage{caption}
\usetikzlibrary{arrows}
\usepackage{capt-of}
\usepackage{hyperref}

\colorlet{darkblue}{blue!50!black}

\hypersetup{
    colorlinks,%
    citecolor=darkblue,%
    filecolor=red,%
    linkcolor=darkblue,%
    urlcolor=magenta,%
    pdfnewwindow=true,%
    pdfstartview={FitH}
}

\newcommand{\p}{\partial}
\newcommand{\e}{\varepsilon}

\newcommand{\R}{{\mathbb R}}
\newcommand{\IP}{{\mathbb P}}

\newcommand{\E}{{\mathbb E}}

\newcommand{\BBB}{\boldsymbol{\mathit B}}

\newcommand{\bE}{\mathbf E}
\newcommand{\bG}{\mathbf G}
\newcommand{\bP}{\mathbf P}
\newcommand{\PPP}{\boldsymbol{\mathit P}}

\newcommand{\SSS}{\boldsymbol{\mathit S}}
\newcommand{\boR}{\boldsymbol{\mathit R}}

\newcommand{\XXX}{\boldsymbol{\mathit X}}

\newcommand{\GGamma}{\boldsymbol\Gamma}

\newcommand{\SE}{\mathscr E}

\newcommand{\mmu}{{\boldsymbol\mu}}

\newcommand{\ttau}{{\boldsymbol\tau}}
\newcommand{\oomega}{{\boldsymbol\omega}}

\newcommand{\zzeta}{{\boldsymbol\zeta}}
\newcommand{\OOmega}{{\boldsymbol\Omega}}

\newcommand{\aA}{{\cal A}}
\newcommand{\BB}{{\cal B}}
\newcommand{\CC}{{\mathscr C}}
\newcommand{\DD}{{\cal D}}
\newcommand{\EE}{{\cal E}}
\newcommand{\FF}{{\cal F}}

\newcommand{\KK}{{\cal K}}
\newcommand{\LL}{{\cal L}}

\newcommand{\PP}{{\cal P}}
\newcommand{\RR}{{\cal R}}

\newcommand{\VV}{{\cal V}}

\newcommand{\XX}{{\cal X}}
\newcommand{\YY}{{\cal Y}}
\newcommand{\ZZ}{{\cal Z}}

\newcommand{\UU}{{\mathcal U}}

\newcommand{\RRR}{\boldsymbol\RR}
\newcommand{\FFF}{\boldsymbol\FF}
\newcommand{\DDD}{\boldsymbol D}

\newcommand{\dd}{{\textup d}}

\newcommand{\PPPP}{{\mathfrak P}}
\newcommand{\BBBB}{{\mathfrak B}}
\newcommand{\MMMM}{{\mathfrak M}}

\newcommand{\nnn}{{\boldsymbol{\mathit n}}}
\newcommand{\uuu}{{\boldsymbol{\mathit u}}}
\newcommand{\vvv}{{\boldsymbol{\mathit v}}}

\newcommand{\supp}{\mathop{\rm supp}\nolimits}
\newcommand{\diver}{\mathop{\rm div}\nolimits}

\newcommand{\esssup}{\mathop{\rm ess\ sup}}

\theoremstyle{plain}
\newtheorem*{mt}{Main Theorem}

\newtheorem{theorem}{Theorem}[section]
\newtheorem{lemma}[theorem]{Lemma}
\newtheorem{proposition}[theorem]{Proposition}
\newtheorem{corollary}[theorem]{Corollary}
\theoremstyle{definition}

\theoremstyle{remark}
\newtheorem{remark}[theorem]{Remark}

\newtheorem*{example*}{Example}

\numberwithin{equation}{section}

\begin{document}
\author{Armen Shirikyan} 
\title{Controllability implies mixing~II.\\
Convergence in the dual-Lipschitz metric}
\date{\small D\'epartement de Math\'ematiques, Universit\'e de Cergy--Pontoise, 
CNRS UMR8088\\
2 avenue Adolphe Chauvin, 95302 Cergy--Pontoise Cedex, France\\ 
E-mail: \href{mailto:Armen.Shirikyan@u-cergy.fr}{Armen.Shirikyan@u-cergy.fr}\\[4pt]
Department of Mathematics and Statistics, McGill University\\ 805 Sherbrooke Street West, Montreal, QC, H3A 2K6, Canada\\[4pt]
Centre de Recherches Math\'ematiques, CNRS UMI3457, Universit\'e de Montr\'eal\\ 
Montr\'eal,  QC, H3C 3J7, Canada}
\maketitle

\begin{abstract}
This paper continues our study of the interconnection between controllability and mixing properties of random dynamical systems. We begin with an abstract result showing that the approximate controllability to a point and a local stabilisation property imply the uniqueness of a stationary measure and exponential mixing in the dual-Lipschitz metric. This result is then applied to the 2D Navier--Stokes system driven by a random force acting through the boundary. A by-product of our analysis is the local exponential stabilisation of the boundary-driven Navier--Stokes system by a regular boundary control.


\end{abstract}
\tableofcontents
\setcounter{section}{-1}

\section{Introduction}
\label{s0}
In the first part of this project~\cite{shirikyan-rms2017}, we studied a class of ordinary differential equations driven by vector fields with random amplitudes and proved that good knowledge of controllability properties ensures the uniqueness of a stationary distribution and exponential convergence to it in the total variation metric. A key property used in that work was the {\it solid controllability\/} from a point, which means, roughly speaking, that we have exact controllability from that point to a ball, and it is stable under small perturbations. In the case of partial differential equations, this property is rarely satisfied, and the aim of this paper is to replace it by a weaker condition of {\it local stabilisation\/} and to prove that it is still sufficient for the uniqueness of a stationary measure, whereas the convergence to it holds in the dual-Lipschitz  metric, which metrises the weak topology.

To be precise, we confine ourselves in the introduction to the main result of the paper on mixing for the 2D Navier--Stokes system driven by a boundary noise. Thus, we consider the problem
\begin{align} 
\p_tu+\langle u,\nabla\rangle u-\nu\Delta u+\nabla p&=0, \quad \diver u=0,
\quad x\in D,\label{0.4}\\
u\bigr|_{\p D}&=\eta, \label{0.5}
\end{align}
where $D\subset\R^2$ is a bounded domain with smooth boundary, $u=(u_1,u_2)$ and~$p$ are unknown velocity field and pressure, $\nu>0$ is the kinematic viscosity, and $\eta=\eta(t,x)$ is a random force that acts through the boundary and has a piecewise independent structure. Namely, we assume that
\begin{equation} \label{0.6}
\eta(t,x)=\sum_{k=1}^\infty I_{[k-1,k)}(t)\eta_k(t-k+1,x),
\end{equation}
where $I_{[k-1,k)}$ is the indicator function of the interval~$[k-1,k)$ and~$\{\eta_k\}$ is a sequence of i.i.d.\ random variables in the space~$L^2([0,1]\times\p D,\R^2)$ that possess some additional properties ensuring the well-posedness of problem~\eqref{0.4}, \eqref{0.5}. We are interested in the problem of mixing for the corresponding random flow. 

Let us formulate our main result informally, omitting some unessential technical details. We set $J=[0,1]$ and $\Sigma=J\times\p D$, and define~$E$ as the space of restrictions to~$\Sigma$ of the time-dependent divergence-free vector fields $u=(u_1,u_2)$ on $J\times D$ that satisfy the inclusions 
\begin{equation} \label{0.11}
u\in L^2(J,H^3), \quad \p_tu\in L^2(J,H^1),
\end{equation}
where $H^k$ stands for the Sobolev space of order~$k\ge0$ on the domain~$D$. An exact description of~$E$ can be found in the paper~\cite{FGH-2002} (see also Section~\ref{s4.4}), where it is shown, in particular, that~$E$ has the structure of a separable Hilbert space. We assume that the  random variables~$\eta_k$ belong to~$E$ almost surely and their law~$\ell$ satisfies the following hypothesis:

\smallskip
{\bf(H)}
{\sl The measure $\ell$ has a compact support in~$E$ and is decomposable in the following sense: there is an orthonormal basis~$\{e_j\}$ in~$E$  such that~$\ell$ can be represented as the tensor product of its  projections~$\ell_j$ to the one-dimensional subspaces spanned by~$e_j$. Moreover, $\ell_j$ has a $C^1$-smooth density with respect to the Lebesgue measure for any~$j\ge1$, and its support contains the origin}. 

\smallskip
For a random variable~$\xi$, we denote by~$\DD(\xi)$ its law, and we write $C(J,L^2)$ for the space of continuous functions on~$J$ with range in the space of square integrable vector fields on~$D$.  The following theorem is a simplified version of the main result of this paper (see Section~\ref{s3.1} for an exact and stronger statement). 

\begin{mt}
Under the above hypotheses, there is a probability measure~$\mmu$ on the space~$C(J,L^2)$ such that any solution~$u(t,x)$ of~\eqref{0.4}--\eqref{0.6} issued from a deterministic initial condition satisfies the inequality
\begin{equation} \label{0.7}
\bigl\|\DD(\uuu_k)-\mmu\bigr\|_L^*
\le C\bigl(\|u(0)\|_{L^2}\bigr)e^{-\gamma k}, \quad k\ge1,
\end{equation}
where $\gamma$ is a positive number not depending on~$u_0$, $\uuu_k$ stands for the restriction of the function~$u(t+k-1,x)$ to the cylinder~$[0,1]\times D$, and $\|\cdot\|_L^*$ denotes the dual-Lipschitz metric over the space $C(J,L^2)$. 
\end{mt}

Let us mention that the problem of mixing for randomly forced PDEs attracted a lot of attention in the last two decades, and the case in which all the determining modes are perturbed is rather well understood. We refer the reader to~\cite{FM-1995,KS-cmp2000,EMS-2001,BKL-2002} for the first achievements and to the book~\cite{KS-book} and the review papers~\cite{bricmont-2002,flandoli-2008,debussche-2013} for a detailed account of the results obtained so far in that situation. On the other hand, there are only a few works dealing with the case when the random noise does not act directly on the determining modes. Namely, Hairer and Mattingly~\cite{HM-2006,HM-2011} studied the 2D Navier--Stokes equations on the torus and the sphere and established the property of exponential mixing, provided that the random perturbation is white in time and contains the first few Fourier modes.  F\"oldes, Glatt-Holtz, Richards, and Thomann~\cite{FGRT-2015} proved a similar result for the Boussinesq system, assuming that a highly degenerate random forcing acts only on the equation for the temperature. In~\cite{shirikyan-asens2015},  the property of exponential mixing was stablished for the 2D Navier--Stokes system perturbed by a space-time localised smooth stochastic forcing. Finally, the recent paper~\cite{KNS-2018} proves a similar result in the situation when random forces are localised in the Fourier space and coloured in time. We also mention the papers~\cite{sinai-1991,EKMS-2000,bakhtin-2007,DV-2015,boritchev-2016,GS-2017,shirikyan-cup2017} devoted to the viscous and inviscid Burgers equation and some scalar conservation laws, whose flow possesses a strong stability property. To the best of my knowledge, the problem of mixing for the Navier--Stokes system with a random perturbation acting through the boundary was not studied in earlier works. 

\smallskip
In conclusion, let us mention that this paper is a part of the programme whose goal is to develop methods for applying the results and tools of the control theory in the study of mixing properties of flows generated by randomly forced evolution equations. It complements the earlier results established in~\cite{AKSS-aihp2007,shirikyan-asens2015,shirikyan-rms2017,KNS-2018} and develops a general framework for dealing with random perturbations acting through the boundary of the domain. 

\smallskip
The paper is organised as follows. In Section~\ref{s1}, we study an abstract discrete-time Markov process in a compact metric space and prove a result on uniqueness of a stationary distribution and its exponential stability. Section~\ref{s3} is devoted to discussing the initial-boundary value problem for the Navier--Stokes system and proving some properties of the resolving operator. The main result of the paper on mixing for the 2D Navier--Stokes system perturbed by a random boundary force is presented in Section~\ref{s5}. The appendix gathers a few auxiliary results used in the main text.

\medskip
{\bf Acknowledgement.}
The author is grateful to A.~Miranville and J.-P.~Puel for useful remarks on boundary-driven Navier--Stokes equations and observability inequalities. This research was carried out within the MME-DII Center of Excellence (ANR-11-LABX-0023-01) and supported by  {\it Agence Nationale de la Recherche\/} through the grant NONSTOPS (ANR-17-CE40-0006-02),  {\it Initiative d'excellence Paris-Seine\/}, and the CNRS PICS {\it Fluctuation theorems in stochastic systems\/}.

\subsubsection*{Notation}
Let~$(X,d)$ be a Polish space, let~$E$ be a separable Banach space, let $J\subset\R$ be a closed interval, and let~$D$ be a bounded domain or a surface in a Euclidean space. In addition to the conventions of~\cite{shirikyan-rms2017}, we use the following notation.

\smallskip
\noindent
$L_b(X)$ is the space of bounded continuous functions  $f:X\to\R$ such that
$$
\|f\|_L:=\|f\|_\infty+\sup_{0<d(u,v)\le 1}\frac{|f(u)-f(v)|}{d(u,v)}<\infty,
$$
where $\|\cdot\|_\infty$ is the usual supremum norm. 

\smallskip
\noindent
$\PP(X)$ stands for the set of probability measures with the {\it dual-Lipschitz metric\/}
$$
\|\mu_1-\mu_2\|_L^*=\sup_{\|f\|_L\le1}|(f,\mu_1)-(f,\mu_2)|,
$$
where the supremum is taken over all function $f\in L_b(X)$ with norm $\le1$. 

\smallskip
\noindent
$L^p(J,E)$ is the space of Borel-measurable functions $f:J\to E$ such that 
$$
\|f\|_{L^p(J,E)}=\biggl(\int_J\|f(t)\|_E^p\,\dd t\biggr)^{1/p}<\infty.
$$
In the case $p=\infty$, this norm is replaced by $\|f\|_{L^\infty(J,E)}=\esssup_{t\in J}\|f(t)\|_E$. If $J\subset\R$ is unbounded, then we write $L_{\mathrm{loc}}^p(J,E)$ for the space of functions $f:J\to E$ whose restriction to any bounded interval $I\subset J$ belongs to $L^p(I,E)$. 

\smallskip
\noindent
$\LL(E,F)$ is the space of continuous linear operators from~$E$ to another Banach space~$F$. This space is endowed with the usual operator norm. 

\smallskip
\noindent
$W^{s,q}(D)$ is the standard Sobolev space of measurable functions $f:D\to\R$ such that $\p^\alpha f\in L^q(D)$ for any multi-index $\alpha$ with $|\alpha|\le s$. In the case $q=2$, we shall write~$H^s(D)$. The norms in these spaces are denoted by $\|\cdot\|_{W^{s,q}}$ and~$\|\cdot\|_s$, respectively. We write $W^{s,q}(D,\R^2)$ and $H^s(D,\R^2)$ for the corresponding spaces of vector functions. 

\smallskip
\noindent
$H_0^s=H_0^s(D)$ is the subspace in~$H^s(D)$ consisting of the functions vanishing on~$\p D$. The corresponding space of vector functions $H_0^s(D,\R^2)$ is defined in a similar way. A function $f\in H_0^1(D)$ extended by zero to a larger domain $D'\supset D$ belongs to $H_0^1(D')$; we tacitly assume that any function in~$H_0^1(D)$ is extended by zero outside~$D$. 

\smallskip
\noindent
$C_i$ and $C_i(\dots)$ denote positive numbers, which may depend on the quantities mentioned in the brackets. 

\section{Mixing in the dual-Lipschitz metric}
\label{s1}
\subsection{Description of the model}
\label{s1.1}
Let us consider the following random dynamical system in a compact metric space~$(X,d)$:
\begin{equation} \label{1.1}
u_k=S(u_{k-1},\eta_k), \quad k\ge1.
\end{equation}
Here $\{\eta_k\}$ is a sequence of i.i.d.\ random variables in a separable Banach space~$E$ and $S:X\times E\to X$ is a continuous mapping. Equation~\eqref{1.1} is supplemented with the initial condition
\begin{equation} \label{1.2}
u_0=u,
\end{equation}
where~$u$ is an $X$-valued random variable independent of~$\{\eta_k\}$. We denote by~$(u_k,\IP_u)$ the discrete-time Markov process associated with~\eqref{1.1} and by $P_k(u,\Gamma)$ its transition function. The Markov operators corresponding to~$P_k(u,\Gamma)$ are denoted by
$$
\PPPP_k:C(X)\to C(X), \quad \PPPP_k^*:\PP(X)\to\PP(X), \quad k\ge0.
$$

Due to the compactness of~$X$, the Markov process~$(u_k,\IP_u)$ has at least one stationary distribution~$\mu$, that is, a probability measure satisfying the equation $\PPPP_1^*\mu=\mu$. In this section, we investigate the question of uniqueness of stationary distribution and its exponential stability in the dual-Lipschitz metric. To this end, we introduce some controllability properties for~\eqref{1.1}. 

\medskip
\underline{\it Approximate controllability to a given point}.
Let  $\hat u\in X$ be a point and let~$\KK\subset E$ be a compact subset. System~\eqref{1.1} is said to be {\it globally approximately controllable\footnote{Note that this concept of approximate controllability is slightly stronger than the one used in~\cite{shirikyan-rms2017}.} to~$\hat u$ by a $\KK$-valued control\/} if for any $\e>0$ there exists $m\ge1$ such that, given any initial point $u\in X$, we can find $\zeta_1^u,\dots,\zeta_m^u\in\KK$ for which 
\begin{equation} \label{1.3}
d\bigl(S_m(u;\zeta_1^u,\dots,\zeta_m^u),\hat u\bigr)\le\e,
\end{equation}
where $S_k(u;\eta_1,\dots,\eta_k)$ denotes the trajectory of~\eqref{1.1}, \eqref{1.2}. 

In~\cite{shirikyan-rms2017}, we imposed the condition of solid controllability, which implies, in particular, the exact controllability to a ball. Here we replace it by a property of local stabilisation.

\medskip
\underline{\it Local stabilisability}.
Let us set $\DDD_\delta=\{(u,u')\in X\times X:d(u,u')\le\delta\}$. We say that~\eqref{1.1} is {\it locally stabilisable\/} if for any $R>0$ and any compact set $\KK\subset E$ there is a finite-dimensional subspace $\EE\subset E$, positive numbers~$C$, $\delta$, $\alpha\le1$, and~$q<1$, and a continuous mapping 
$$
\varPhi:\DDD_\delta\times B_E(R)\to \EE, \quad (u,u',\eta)\mapsto \eta', 
$$
which is continuously differentiable in~$\eta$ and satisfies the following inequalities for any $(u,u')\in \DDD_\delta$:
\begin{align}
\sup_{\eta\in B_E(R)}
\bigl(\|\varPhi(u,u',\eta)\|_E+\|D_\eta\varPhi(u,u',\eta)\|_{\LL(E)}\bigr)
&\le C\,d(u,u')^\alpha,\label{1.5}\\
\sup_{\eta\in \KK}d\bigl(S(u,\eta),S(u',\eta+\varPhi(u,u',\eta))\bigr)
&\le q\,d(u,u').\label{1.6}
\end{align}

Finally, concerning the random variable~$\eta_k$, we shall assume that their law~$\ell$ has a compact support $\KK\subset E$ and is {\it decomposable\/} in the following sense. There are two sequences of closed subspaces~$\{F_n\}$ and~$\{G_n\}$ in~$E$ possessing the two properties below:
\begin{itemize}
\item[\bf(a)]
$\dim F_n<\infty$ and $F_n\subset F_{n+1}$ for any $n\ge1$, and the vector space $\cup_nF_n$ is dense in~$E$.
\item[\bf(b)]
$E$ is the direct sum of~$F_n$ and~$G_n$, the norms of the corresponding projections~${\mathsf P}_n$ and~${\mathsf Q}_n$ are bounded uniformly in~$n\ge 1$, and the measure~$\ell$ can be written as the product of its projections ${\mathsf P}_{n*}\ell$ and~${\mathsf Q}_{n*}\ell$ for any $n\ge1$. 
\end{itemize}

\subsection{Uniqueness and exponential mixing}
\label{s1.3}
From now on, we assume that the phase space~$X$ is a compact subset of a Banach space~$H$, endowed with a norm~$\|\cdot\|$. We shall say that a stationary measure $\mu\in\PP(X)$ for $(u_k,\IP_u)$ is {\it exponentially mixing\/} (in the dual-Lipschitz metric) if there are positive numbers~$\gamma$ and~$C$ such that 
\begin{equation} \label{4.12}
\|\PPPP_k^*\lambda-\mu\|_L^*\le Ce^{-\gamma k}
\quad\mbox{for $k\ge0$, $\lambda\in\PP(X)$}. 
\end{equation}
The following result provides an analogue of Theorem~1.1 in~\cite{shirikyan-rms2017} for the case when the property of solid controllability is replaced by local stabilisability.

\begin{theorem} \label{t1.2}
Suppose that $S:H\times E\to H$ is a $C^1$-smooth mapping such that $S(X\times \KK)\subset X$, and system~\eqref{1.1} with phase space~$X$ is locally stabilisable and globally approximately controllable to some point $\hat u\in X$ with a $\KK$-valued control. Let us assume, in addition, that the law~$\ell$ of~$\eta_k$  is decomposable, and the measures~${\mathsf P}_{n*}(\ell)$ possess $C^1$-smooth densities~$\rho_n$ with respect to the Lebesgue measure on~$F_n$. Then~\eqref{1.1} has a unique stationary measure~$\mu\in\PP(X)$, which is exponentially mixing.
\end{theorem}

\begin{proof}
We first outline the scheme\footnote{The key coupling construction of this proof goes back to~\cite{KS-cmp2001} (see Section~3).} of the proof, which is based on an application of Theorem~\ref{t4.2}. To this end, we shall construct an extension $(\uuu_k,\bP_\uuu)$ for the Markov process~$(u_k,\IP_u)$ associated with~\eqref{1.1} such that the squeezing and recurrence properties hold. 

Let us write~$\XXX=X\times X$ and, given a number $\delta>0$, denote
$$
\BBB=\{\uuu=(u,u')\in \XXX:\|u-u'\|\le\delta\}. 
$$
Suppose we can construct a probability space~$(\Omega,\FF,\IP)$ and measurable functions $\RR,\RR':\XXX\times\Omega\to X$ such that the following three properties hold for any $\uuu=(u,u')\in\XXX$:
\begin{itemize}
\item[\bf(a)] 
The pair $(\RR(\uuu,\cdot),\RR'(\uuu,\cdot))$ is a coupling for $(P_1(u,\cdot),P_1(u',\cdot))$. 
\item[\bf(b)] 
If $\uuu\notin\BBB$, then the random variables $\RR(\uuu,\cdot)$ and $\RR'(\uuu,\cdot)$ are independent. 
\item[\bf(c)] 
If $\uuu\in\BBB$, then 
\begin{equation} \label{1.14}
\IP\{\|\RR(\uuu)-\RR'(\uuu)\|>r\|u-u'\|\}\le C\,\|u-u'\|^\alpha,
\end{equation}
where  $r<1$, $C$, and $\alpha\le1$ are positive numbers not depending on~$\uuu$. 
\end{itemize}
In this case, the discrete-time Markov process $(\uuu_k,\bP_\uuu)$ with the time-$1$ transition function
\begin{equation} \label{1.15}
\bP_1(\uuu,\Gamma)
=\IP\bigl\{(\RR(\uuu,\cdot),\RR'(\uuu,\cdot))\in\Gamma\bigr\}, 
\quad \uuu\in\XXX, \quad \Gamma\in\BB(\XXX),
\end{equation}
is an extension for~$(u_k,\IP_u)$ that satisfies the recurrence and squeezing properties of Theorem~\ref{t4.2} (see Steps~1 and~2 below), so that we can conclude. 

The construction of~$\RR$ and~$\RR'$ is trivial for $\uuu\notin\BBB$: it suffices to take two independent $E$-valued random variables~$\eta$ and~$\eta'$ with the law~$\ell$ and to define
\begin{equation} \label{1.16}
\RR(\uuu)=S(u)+\eta, \quad \RR(\uuu)=S(u')+\eta',\quad \uuu=(u,u'). 
\end{equation}
The key point is  the construction of the pair~$\RRR=(\RR,\RR')$ when $\uuu\in\BBB$ and the proof of~\eqref{1.14}. It is based on an estimate of a cost function (Lemma~\ref{l4.6}) and an abstract result on the existence of measurable coupling associated with a cost (Proposition~\ref{p4.5}). We now turn to a detailed proof, which is divided into three steps.

\medskip
{\it Step 1: Recurrence}. 
Suppose we have constructed a pair~$\RRR=(\RR,\RR')$ satisfying properties (a)--(c) given above. Let us show that the Markov process~$(\uuu_k,\bP_\uuu)$ with the transition function~\eqref{1.15} possesses the recurrence property of Theorem~\ref{t4.2}. 

To this end, we first recall a standard construction of the Markov family with the transition function~\eqref{1.15}. Let us define $(\OOmega,\FFF,\bP)$ as the tensor product of countably many copies  the probability space $(\Omega,\FF,\IP)$ on which the pair~$(\RR,\RR')$ is defined. We shall denote by $\oomega=(\omega_1,\omega_2,\dots)$ the points of~$\OOmega$ and write $\omega^{(k)}=(\omega_1,\dots,\omega_k)$. Let us define a family $\{\RRR_k(\uuu),k\ge0,\uuu\in\XXX\}$ recursively by the relation 
\begin{equation} \label{1.17}
\RRR_k(\uuu,\oomega)
=\bigl(\RR_k(\uuu,\oomega),\RR_k'(\uuu,\oomega)\bigr)
=\RRR\bigl(\RRR_{k-1}(\uuu,\oomega),\omega_k\bigr), \quad k\ge1,
\end{equation}
which implies, in particular, that~$\RRR_k$ depends only on~$\omega^{(k)}$. It is straightforward to check that the sequences $\{\RRR_k(\uuu)\}_{k\ge0}$ defined on the probability space $(\OOmega,\FFF,\bP)$ form a Markov process with the transition function~\eqref{1.15}. 

To prove inequality~\eqref{4.10} for the first hitting time~$\ttau=\ttau(\BBB)$ of the set~$\BBB$, it suffices to show that 
\begin{equation} \label{1.18}
\bP\{\RRR_m(\uuu,\cdot)\in\BBB\}\ge p
\quad\mbox{for any $\uuu\in\XXX$},
\end{equation}
where the integer $m\ge1$ and the number~$p>0$ do not depend on~$\uuu$. Indeed, once this inequality is established, a simple application of the Markov property will imply that 
$$
\bP\{\RRR_{km}(\uuu,\cdot)\notin\BBB\mbox{ for $1\le k\le j$}\}\le (1-p)^j
\quad\mbox{for any $\uuu\in\XXX$, $j\ge1$}. 
$$
The required inequality follows now from the Borel--Cantelli lemma. 

Inequality~\eqref{1.18} would be a simple consequence of the approximate controllability to a given point if the processes $\{\RR_k(\uuu),k\ge0\}$ and $\{\RR_k'(\uuu),k\ge0\}$ were independent. However, this is not the case, and we have to proceed differently. We shall need the following auxiliary results established at the end of this section. Given an integer $k\ge1$, let~$X^k$ be the direct product of~$k$ copies of~$X$ and let $T_k:=\{\ttau\ge k\}$. 

\begin{lemma} \label{l1.3}
For any $m\ge0$, the random variables $\{\RR_k(\uuu),k=0,\dots,m\}$ and $\{\RR_k'(\uuu),k=0,\dots,m\}$ valued in~$X^{m+1}$ are independent on the set $T_m$; that is, for any $\Gamma,\Gamma'\in\BB(X^{m+1})$, we have
\begin{multline} \label{1.19}
\bP_\uuu\bigl\{(\RR_0,\dots,\RR_m)\in\Gamma,
(\RR_0',\dots,\RR_m')\in\Gamma'\,|\,T_m\bigr\}\\
=\bP_\uuu\bigl\{(\RR_0,\dots,\RR_m)\in\Gamma\,|\,T_m\bigr\}
\bP_\uuu\bigl\{(\RR_0',\dots,\RR_m')\in\Gamma'\,|\,T_m\bigr\}.
\end{multline}
\end{lemma}

\begin{lemma} \label{l1.4}
There is $C_1>0$ such that, for any $\uuu\in\BBB$, we have
\begin{equation} \label{1.20}
\bP\bigl\{\|\RR_k(\uuu)-\RR_k'(\uuu)\|\le r^k\|u-u'\|\mbox{ for all }k\ge0\bigr\}
\ge 1-C_1\|u-u'\|^\alpha,
\end{equation}
where $\alpha$ and~$r$ are the numbers entering~\eqref{1.14}.  
\end{lemma}

Taking these lemmas for granted, we prove~\eqref{1.18}. Let $m\ge1$ be the integer entering the hypothesis of approximate controllability with $\e=\delta/2$; see~\eqref{1.3}. We claim that~\eqref{1.18} holds with this choice of~$m$ and a sufficiently small~$p>0$. To prove this, we write
\begin{equation} \label{1.21}
\bP\{\RRR_m(\uuu)\in\BBB\}
=\bP\bigl(\{\RRR_m(\uuu)\in\BBB\}\cap T_m^c\bigr)
+\bP\bigl(\{\RRR_m(\uuu)\in\BBB\}\cap T_m\bigr).
\end{equation}
In view of the strong Markov property, we have
\begin{align}
\bP\bigl(\{\RRR_m(\uuu)\in\BBB\}\cap T_m^c\bigr)
&=\bE\bigl(I_{\BBB}(\RRR_m(\uuu))I_{T_m^c}\bigr)\notag\\
&=\bE\bigl(I_{T_m^c}\bE\{I_{\BBB}(\RRR_m(\uuu))\,|\,\FFF_\ttau\}\bigr)
\notag\\
&=\bE\bigl(I_{T_m^c}\bP_{\RRR_{\ttau}(\uuu)}\{\RRR_k\in\BBB\}|_{k=m-\ttau}\bigr). 
\label{1.22}
\end{align}
Since $\RRR_\ttau\in\BBB$, it follows from~\eqref{1.20} that the probability on the right-hand side of~\eqref{1.22} is bounded below by $1-C_1\delta^\alpha$. Combining this with~\eqref{1.21}, we see that
\begin{equation} \label{1.23}
\bP\{\RRR_m(\uuu)\in\BBB\}
\ge (1-C_1\delta^\alpha)\bP_\uuu(T_m^c)
+\bP\{\RRR_m(\uuu)\in\BBB\,|\,T_m\bigr\}\bP(T_m). 
\end{equation}
Let us fix a small number~$\nu>0$ (it will be chosen below) and assume first that $\bP(T_m^c)\ge\nu$. In this case, we obtain
$$
\bP\{\RRR_m(\uuu)\in\BBB\}\le (1-C_1\delta^\alpha)\,\nu>0,
$$
provided that $\delta>0$ is sufficiently small. Thus, we can assume that $\bP(T_m^c)\le\nu$, so that $\bP(T_m)\ge1-\nu$. Denoting by~$Q\subset X$ the closed ball of radius~$\delta/2$ centred at $\hat u$ (where $\hat u\in X$ is the point entering the hypothesis of approximate controllability) and using Lemma~\ref{l1.3}, we can write
\begin{align}
\bP\{\RRR_m(\uuu)\in\BBB\,|\,T_m\bigr\}
&\ge \bP\{\RR_m(\uuu)\in Q, \RR_m'(\uuu)\in Q\,|\,T_m\bigr\}
\notag\\
&=\bP\{\RR_m(\uuu)\in Q\,|\,T_m\bigr\}\bP\{\RR_m'(\uuu)\in Q\,|\,T_m\bigr\}. 
\label{1.24} 
\end{align}
Suppose we found $\varkappa>0$ such that 
\begin{equation} \label{1.25}
\bP\{\RR_m(\uuu)\in Q\,|\,T_m\bigr\}\ge\varkappa, \quad
\bP\{\RR_m'(\uuu)\in Q\,|\,T_m\bigr\}\ge\varkappa
\quad\mbox{for all $\uuu\in\XXX$}. 
\end{equation}
In this case, combining~\eqref{1.23}--\eqref{1.25}, we obtain
$$
\bP\{\RRR_m(\uuu)\in\BBB\}\ge \varkappa^2(1-\nu). 
$$
Thus, it remains to establish inequalities~\eqref{1.25}. We confine ourselves to the first one, since the proof of the other is similar. 

The approximate controllability to~$\hat u$ combined with a standard argument implies that 
$$
\beta:=\inf_{\uuu\in\XXX}\bP\{\RR_m(\uuu)\in Q\}>0. 
$$
Assuming that the parameter $\nu>0$ fixed above is smaller than~$\beta$, for any $\uuu\in\XXX$ we derive
\begin{align*}
\beta\le \bP\{\RR_m(\uuu)\in Q\bigr\}
&\le \bP\{\RR_m(\uuu)\in Q\,|\,T_m\bigr\}\bP(T_m)+\bP(T_m^c)\\
&\le \bP\{\RR_m(\uuu)\in Q\,|\,T_m\bigr\}+\nu,
\end{align*}
whence we conclude that~\eqref{1.25} holds with $\varkappa=\beta-\nu$. 

\medskip
{\it Step 2: Squeezing}. 
We now prove that~$(\uuu_k,\bP_\uuu)$ satisfies the squeezing property of Theorem~\ref{t4.2}. Namely, we claim that inequalities~\eqref{4.11} hold for the Markov time
$$
\sigma(\uuu)=\min\{k\ge0:\|\RR_k(\uuu)-\RR_k(\uuu)\|> r^k \delta\},
$$
provided that $\delta>0$ is sufficiently small. 

We first note that, if $\delta\le 1$ and $\uuu\in\BBB$, then
$$
\{\sigma(\uuu)=+\infty\}\supset
\bigl\{\|\RR_k(\uuu)-\RR_k(\uuu)\|\le r^k\|u-u'\|\mbox{ for }k\ge0\bigr\}. 
$$
Hence, the first inequality in~\eqref{4.11} follows immediately from~\eqref{1.20}, provided that $C_1\delta^\alpha<1$. 

Let us prove the second inequality in~\eqref{4.11}. To this end, note that, for $k\ge0$, we have
$$
\{\sigma(\uuu)=k+1\}\subset
\bigl\{\|\RR_k(\uuu)-\RR_k(\uuu)\|\le\delta r^k,
\|\RR_{k+1}(\uuu)-\RR_{k+1}(\uuu)\|>\delta r^{k+1}\bigr\}. 
$$
Applying the Markov property and using~\eqref{1.14}, we derive
\begin{align*}
\bP\{\sigma(\uuu)=k+1\}
&\le \bE\Bigl(I_{G_k(\uuu)}\bP\bigl\{\|\RR_1(\vvv)-\RR_1'(\vvv)\|
>\delta r^{k+1}\bigr\}\big|_{\vvv=\RRR_k(\uuu)}\Bigr)\\
&\le C(\delta r^k)^\alpha\bP\{G_k(\uuu)\}, 
\end{align*}
where we set $G_k(\uuu)=\{\|\RR_k(\uuu)-\RR_k(\uuu)\|\le\delta r^k\}$. Choosing $\delta\in(0,1]$ so small that $C\delta^\alpha\le 1$, we see that
$$
\bP\{\sigma(\uuu)=k\}\le r^{\alpha(k-1)}\quad\mbox{for any $k\ge1$}. 
$$
It follows that the second inequality in~\eqref{4.11} holds for $\delta_2<\alpha \ln r^{-1}$. 

\medskip
{\it Step 3: Construction of~$(\RR,\RR')$}. To complete the proof, it remains to construct the pair~$(\RR,\RR')$ and to prove~\eqref{1.14}. To this end, we shall use Propositions~\ref{p4.4} and~\ref{p4.5}. 

Let us consider the pair of probability measures $(P_1(u,\cdot),P_1(u',\cdot))$ on~$X$ depending on the parameter $\uuu\in\BBB$. Fix any number $r\in(q,1)$, where~$q\in(0,1)$ is the constant in~\eqref{1.6}, and define the function $\e(\uuu)=r\|u-u'\|$. Applying Proposition~\ref{p4.5} with $\theta=q/r$, we can construct a pair of random variables $(\RR(\uuu,\cdot),\RR'(\uuu,\cdot))$ on the same probability space $(\Omega,\FF,\IP)$ such that (see~\eqref{4.32}) 
\begin{equation} \label{1.26}
\IP\{\|\RR(\uuu,\cdot)-\RR'(\uuu,\cdot)\|>r\|u-u'\|\}
\le C_{q\|u-u'\|}\bigl(P_1(u,\cdot),P_1(u',\cdot)\bigr),
\end{equation}
where $\uuu=(u,u')\in\BBB$. We now use Proposition~\ref{p4.4} and Lemma~\ref{l4.6} to estimate the right-hand side of this inequality. 

Let us fix $R>0$ so large that $\KK\subset B_E(R)$. In view of local stabilisability, one can find a finite-dimensional subspace $\EE\subset E$ and a mapping $\varPhi:\BBB\times B_E(R)\to\EE$ such that~\eqref{1.5} and~\eqref{1.6} hold. The measures $P_1(u,\cdot)$ and $P_1(u',\cdot)$ coincide with the laws of the random variables $S(u,\xi)$ and $S(u',\xi)$ defined on the probability space $(E,\BB(E),\ell)$, where $\xi:E\to E$ is the identity mapping. By Lemma~\ref{l4.6}, in which $\e=\e(\uuu)=q\|u-u'\|$, we have
\begin{equation} \label{1.27}
C_{q\|u-u'\|}\bigl(P_1(u,\cdot),P_1(u',\cdot)\bigr)
\le 2\|\ell-\varPsi_*(\ell)\|_{\text{var}},
\end{equation}
where $\varPsi(\zeta)=\zeta+\varPhi(u,u',\zeta)$. Using now Proposition~\ref{p4.4} and inequality~\eqref{1.5}, we see that 
$$
\|\ell-\varPsi_*(\ell)\|_{\text{var}}\le C_1\|u-u'\|^\alpha,
$$
where $C_1>0$ does not depend on~$u$ and~$u'$. Combining this inequality with~\eqref{1.27} and~\eqref{1.26}, we arrive at the required inequality~\eqref{1.14}. The proof of Theorem~\ref{t1.2} is complete. 
\end{proof}

\begin{proof}[Proof of Lemma~\ref{l1.3}]
Let us note that the event~$T_m$ can be written as
$$
T_m=\{\ttau\ge m\}
=\bigl\{\bigl(\RRR_0(\uuu),\dots,\RRR_m(\uuu)\bigr)\in \bG\bigr\},
$$ 
where $\bG=\BBB^c\times\cdots\times\BBB^c\times\XXX$, and the set~$\BBB^c$ is  repeated~$m$ times. Furthermore, on the set~$T_m$, we have
$$
\bigl(\RR_0(\uuu),\dots,\RR_m(\uuu)\bigr)=F_u(\zzeta), \quad
\bigl(\RR_0'(\uuu),\dots,\RR_m'(\uuu)\bigr)=F_{u'}(\zzeta'),
$$
where $F_v:E^m\to X^{m+1}$ is a continuous function depending on~$v\in X$, and~$\zzeta$ and~$\zzeta'$ are independent $E^m$-valued random variables. It follows that~\eqref{1.19} is equivalent to
\begin{multline*}
\bP\bigl\{F_u(\zzeta)\in\Gamma,F_{u'}(\zzeta')\in\Gamma'\,|\,
(F_u(\zzeta),F_{u'}(\zzeta'))\in\bG\bigr\}\\
=\bP\bigl\{F_u(\zzeta)\in\Gamma\,|\,
(F_u(\zzeta),F_{u'}(\zzeta'))\in\bG\bigr\}
\bP\bigl\{F_{u'}(\zzeta')\in\Gamma'\,|\,
(F_u(\zzeta),F_{u'}(\zzeta'))\in\bG\bigr\}.
\end{multline*}
This relation is easily checked for sets~$\bG\in\BB(\XXX^{m+1})$ of the form $\bG=G\times G'$, where $G,G'\in\BB(X^{m+1})$. The general case can be derived with the help of the monotone class lemma. 
\end{proof}

\begin{proof}[Proof of Lemma~\ref{l1.4}]
Inequality~\eqref{1.14} implies that 
\begin{equation} \label{1.51}
\bP\bigl\{\|\RR(\uuu)-\RR'(\uuu)\|\le r\|u-u'\|\bigr\}\ge1-C\|u-u'\|^\alpha
\quad\mbox{for $\uuu=(u,u')\in\BBB$}.
\end{equation}
Let us define the sets
\begin{equation} \label{1.52}
\Gamma_n(\uuu)
=\bigl\{\|\RR_k(\uuu)-\RR_k'(\uuu)\|\le r\|\RR_{k-1}(\uuu)-\RR_{k-1}'(\uuu)\|
\mbox{ for }1\le k \le n\bigr\}.
\end{equation}
Combining~\eqref{1.51} with the Markov property, for $\uuu\in\BBB$ we derive 
\begin{align} 
\bP\bigl(\Gamma_n(\uuu)\bigr)
&=\bE\bigl(I_{\Gamma_n(\uuu)}
\bP\bigl\{\|\RR_n(\uuu)-\RR_n'(\uuu)\|\le r\|\RR_{n-1}(\uuu)-\RR_{n-1}'(\uuu)\|\,
|\,\FFF_{n-1}\bigr\}\bigr)
\notag\\
&\ge\bE\bigl(I_{\Gamma_n(\uuu)}
(1-C\|\RR_{n-1}(\uuu)-\RR_{n-1}'(\uuu)\|^\alpha)\bigr\}.
\label{1.53}
\end{align}
It follows from~\eqref{1.52} that, on the set~$\Gamma_n(\uuu)$, we have
$$
\|\RR_k(\uuu)-\RR_k'(\uuu)\|\le r^k\|u-u'\|\quad\mbox{for $0\le k\le n$}. 
$$
Substituting this into~\eqref{1.53}, we derive
$$
\bP\bigl(\Gamma_n(\uuu)\bigr)
\le \bigl(1-Cr^{\alpha(n-1)}\|u-u'\|^\alpha\bigr)\bP\bigl(\Gamma_{n-1}(\uuu)\bigr).
$$
Iteration of  this inequality results in
\begin{equation}
\bP\bigl(\Gamma_n(\uuu)\bigr)
\ge \prod_{k=0}^{n-1} \bigl(1-Cr^{\alpha k}\|u-u'\|^\alpha\bigr)
\ge 1-2C(1-r^\alpha)^{-1}\|u-u'\|^\alpha,\label{1.54}
\end{equation}
provided that $\uuu=(u,u')\in\BBB$ and the number~$\delta>0$ is sufficiently small.
The left-hand side of~\eqref{1.20} is minorised by the probability of $\cap_{n\ge1}\Gamma_n(\uuu)$, and therefore the required estimate follows from~\eqref{1.54}.
\end{proof}

\section{Initial-boundary value problem for the Navier--Stokes system}
\label{s3}
In this section, we study the Cauchy problem for the 2D Navier--Stokes equations, supplemented with an inhomogeneous boundary condition. This type of results are rather well known in the literature (e.g., see the paper~\cite{FGH-2002} and the references therein), so that some of the proofs are only sketched. The additional properties of the resolving operator that are established in this section will be important when proving the exponential mixing of the random flow associated with the 2D Navier--Stokes system.

\subsection{Resolving operator for the Cauchy problem}
\label{s3.0}
Let $D\subset\R^2$ be a bounded domain with infinitely smooth boundary~$\p D$ such that
\begin{equation} \label{3.2}
D=\widetilde D\setminus\biggl(\,\bigcup_{i=1}^m\overline{D}_i\biggr),
\end{equation}
\begin{wrapfigure}{r}{0.38\textwidth}
\vspace{-33pt}
\begin{center}
\begin{tikzpicture}
\draw[domain=0:360,smooth,samples=100] plot (\x:{2+sin(3*\x)/4+cos(2*\x)/6+cos(3*\x+25)/7});
\draw[fill=gray!30!white] (0,1) circle(0.4) node {$D_1$}; 
\draw[fill=gray!30!white] (1,-0.1) circle(0.35) node {$D_2$}; 
\draw[fill=gray!30!white] (-0.5,-0.3) circle(0.45) node {$D_3$}; 
\draw (0,-1.3) node {$D$}; 
\end{tikzpicture}
\end{center}
\captionof{figure}{The domain~$D$\label{pic1}}
\end{wrapfigure}
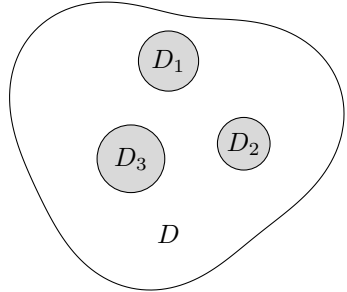%
where $D_1,\dots,D_m$ and~$\widetilde D$ are simply-connected domains  in~$\R^2$ satisfying the properties $\overline{D}_i\subset \widetilde D$ and $\overline{D}_i\cap \overline{D}_j=\varnothing$ for all $i\ne j$; see Figure~\ref{pic1}. Thus, $D$ is a ``domain with~$m$ holes.''  

We consider problem~\eqref{0.4}, \eqref{0.5}, supplemen\-ted with the initial condition
\begin{equation} \label{3.1}
u(0,x)=u_0(x). 
\end{equation}
Let us recall some well-known results on the initial-boundary value problem, specifying, in particular, the functional spaces for~$u_0$ and~$\eta$. 

Let us set $J=[0,1]$ and $\Sigma=J\times\p D$ and define $G\subset L^2(\Sigma,\R^2)$ as the space of functions that are restrictions to~$\Sigma$ of divergence-free vector fields~$u(t,x)$ in~$J\times D$ such that~\eqref{0.11} holds. The space~$G$ is endowed with the natural norm of the quotient space, and its explicit description is given in Section~\ref{s4.4}. Namely, we have 
\begin{equation} \label{3.3}
G=\Bigl\{v\in L^2(J,H^{5/2}): 
\p_tv \in L^2(J,H^{1/2}), \int_{\p D}\langle v(t),\nnn_x\rangle
\,\dd\sigma=0\mbox{ for }t\in J\Bigr\},
\end{equation}
where $H^s=H^s(\p D,\R^2)$ denotes the Sobolev space of order~$s\ge0$ and~$\nnn_x$ stands for the outward unit normal to~$\p D$ at the point~$x$. We shall also need a larger space~$G_s$ (with $3/2\le s\le 2$) defined as the space of functions $v\in L^2(J,H^{s+1/2})$ such that $\p_tv\in L^2(J,H^{s-3/2})$ and  $\int_{\p D}\langle v(t),\nnn_x\rangle\,\dd\sigma=0$ for $t\in J$, so that $G=G_2$. Let us introduce the space 
\begin{equation} \label{XX}
\XX=\bigl\{v\in L^2(J,H_\sigma^{2}): \p_t v\in L^2(J,L_\sigma^2)\bigr\},	
\end{equation}
where $H_\sigma^s=H_\sigma^s(D,\R^2)$ stands for the space of divergence-free vector fields on~$D$ with components belonging to the Sobolev space of order~$s\ge0$, and~$L_\sigma^2=H_\sigma^0$. The following proposition provides a sufficient condition for the well-posedness of the initial-boundary value problem for the Navier--Stokes equations and establishes some properties of the resolving operator. Since these results are important for what follows, we give rather detailed proofs. 

\begin{proposition} \label{p3.1}
For any initial function~$u_0\in V:=H_\sigma^1\cap H_0^1$ and any boundary function $\eta\in G$ vanishing at $t=0$, problem~\eqref{0.4}, \eqref{0.5}, \eqref{3.1} has a unique solution $u\in \XX$, and the resolving operator $\SSS:V\times G\to \XX$ taking $(u_0,\eta)$ to $u(t,x)$ is infinitely differentiable in the Fr\'echet sense. Moreover, the following properties hold.
\begin{itemize}
\item[\bf(a)] 
The mapping~$\SSS$ is continuous and is bounded on bounded subsets. Moreover, its restriction to any ball in~$V\times G$ is Lipschitz-continuous from\,\footnote{The space $L_\sigma^2\times G_s$ is certainly not optimal for the validity of Lipschitz continuity of~$\SSS$. However, it is sufficient for our purposes.}~$L_\sigma^2\times G_s$ to~$C(J,L^2)$ for any $s\in(\frac32,2]$.  
\item[\bf(b)] 
Suppose, in addition, that~$\eta$ belongs to the space
\begin{equation} \label{3.38}
G(\tau_0)=\{\xi\in G: \supp\xi\subset[\tau_0,1]\times\p D\},
\end{equation}
where $\tau_0>0$ is a number, and for $\tau\in(0,1)$, let $\SSS^\tau(u_0,\eta)$ be the restriction of~$\SSS(u_0,\eta)$ to~$J_\tau=[\tau,1]$. Then $\SSS^\tau(u_0,\eta)\in C(J_\tau,W^{1,q})$ for any\,\footnote{A finer analysis shows that this inclusion is valid for any $q\ge2$; however, we do not need that result.} $q\in[2,4)$, and the corresponding norm remains bounded as $(u_0,\eta)$ varies in a bounded subset of~$V\times G(\tau_0)$. 
\end{itemize}
\end{proposition}

Iterating the mapping~$\SSS$ constructed in Proposition~\ref{p3.1}, we obtain a global solution~$u(t,x)$ of problem~\eqref{0.4}, \eqref{0.5}, \eqref{3.1} for any initial function~$u_0\in V$ and boundary forcing $\eta(t,x)$ whose restriction to any interval~$J_k=[k-1,k]$ belongs to the space defined in~\eqref{3.3} with $J=J_k$ and vanishes at the endpoints. We shall write~$S_t(u_0,\eta)$ for the value of~$u$ at time~$t$, so that $S_t(u_0,\eta)\in H_\sigma^1$ for $t\ge0$ and $S_k(u_0,\eta)\in V$ for any integer~$k\ge0$.

\begin{remark} \label{r3.2}
For $\tau\in(0,1)$, let~$\XX^\tau$ be the space~$\XX$ considered on~$J_\tau=[\tau,1]$. 
Analysing the proof given below, it is easy to see that, in Proposition~\ref{p3.1}, one can take any initial condition~$u_0$ belonging to the space 
\begin{equation} \label{3.30}
H=\{u\in L^2(D,\R^2):\diver u=0\mbox{ in $D$}, 
\langle u,\nnn\rangle=0 \mbox{ on $\p D$}\}. 
\end{equation}
In this case, the solution will be less regular in an arbitrarily small neighbourhood of $t=0$. However, it will belong to~$\XX^\tau$ for any $\tau\in(0,1)$, the operator~$\SSS^\tau$ will be continuous and bounded from~$H\times G$ to~$\XX^\tau$, and property~(b) will be true with~$V$ replaced by~$H$ .  
\end{remark}

\subsection{Proof of Proposition~\ref{p3.1}}
The uniqueness of a solution in the space~$\XX$ is standard and can be proved by taking the inner product in~$L^2$ of the equation for the difference $u=u_1-u_2$ between two solutions with~$u$. Let us outline the proof of the existence of a solution and the regularity of the resolving operator. 

\smallskip
We seek a solution of~\eqref{0.4}, \eqref{0.5}, \eqref{3.1} in the form $u=\zeta+v$, where $\zeta=Q\eta$ is an extension of~$\eta$ to the cylinder $J\times D$; see Proposition~\ref{p4.7} for the definition of the operator~$Q$. Then $v(t,x)$ must satisfy the equations
\begin{align} 
\p_tv+\langle \zeta+v,\nabla\rangle (\zeta+v)-\nu\Delta v+\nabla p&=(\nu\Delta-\p_t)\zeta, 
\quad \diver v=0,\label{3.5}\\
v\bigr|_{\Sigma}=0, \quad v(0,x)&=u_0(x),\label{3.6}
\end{align}
where we used the fact that $(Q\eta)(0)=0$ if $\eta(0)=0$. 
We claim that problem \eqref{3.5}, \eqref{3.6} has a unique solution~$v$ in the space $\YY=\{u\in\XX:u\in C(J,V)\}$. Indeed, Eq.~\eqref{3.5} is a Navier--Stokes type system involving an addition function $\zeta\in\XX_2$; see~\eqref{4.61} for the definition of the spaces~$\XX_k$. The unique solvability of~\eqref{3.5}, \eqref{3.6} in~$\YY$ can be established by repeating the corresponding proof for the 2D Navier--Stokes system; e.g., see Section~5 in~\cite[Chapter~17]{taylor1996}. Thus, we can define the mapping $\SSS(u_0,\eta)=Q\eta+v$, which gives a unique solution of \eqref{3.5}, \eqref{3.6}. Moreover, application of the implicit function theorem shows that the resolving operator taking $(u_0,\eta)\in V\times G$ to~$v\in\YY$ is infinitely differentiable and, hence, so is~$\SSS$; see Theorem~2.4 in~\cite{kuksin-1982} for the more complicated 3D case. To complete the proof, it remains to establish properties~(a) and~(b).

\medskip
{\it Proof of~{\rm(a)}}. Since~$Q:G\to\XX_2$ is a continuous linear operator that can be extended to a continuous operator from~$G_s$ to~$\XX_s$ for any $s>3/2$ (see Remark~\ref{r3.8}), it suffices to show that the mapping $\boR:V\times \XX_2\to\YY$ taking $(u_0,\zeta)$ to~$v$ is continuous, is bounded on bounded subsets, and satisfies the inequality
\begin{equation} \label{LC}
\|\boR(u_{01},\zeta_1)-\boR(u_{02},\zeta_2)\|_{C(J,L^2)}
\le C_s(R)\bigl(\|u_{01}-u_{02}\|+\|\zeta_1-\zeta_2\|_{G_s}\bigr),
\end{equation}
where $u_{0i}\in V$ and $\zeta_i\in\XX_2$ are arbitrary functions whose norms are bounded by a number~$R$.  

We first derive an estimate for the norm of~$v$ in the space $C(J,H)\cap L^2(J,V)$. Denoting $h=(\nu\Delta-\p_t)\zeta$, taking the scalar product in~$L^2$ of the first equation in~\eqref{3.5} with~$2v$, and carrying out some standard transformations, we derive
\begin{equation} \label{3.7}
\p_t\|v\|^2+2\nu\|\nabla v\|^2
=2(h,v)-2\bigl(\langle \zeta+v,\nabla\rangle\zeta,v\bigr). 
\end{equation}
When~$\eta$ varies in a bounded set in~$G$, the norms of~$\zeta$ and~$h$ in the spaces~$\XX_2$ and~$L^2(J,H^1)$, respectively, remain bounded by a number~$M$. Furthermore, H\"older's inequality and Sobolev's embeddings enable one to show that
\begin{align}
|(h,v)|&\le \frac{\nu}{4}\|\nabla v\|^2+C_1\|h\|^2,
\label{3.42}\\
\bigl|(\langle \zeta,\nabla\rangle\zeta,v)\bigr|
&\le\bigl|(\langle \zeta,\nabla\rangle v,\zeta)\bigr|
\le C_1\|\nabla v\|\,\|\zeta\|_{L^4}^2\notag\\
&\le \frac{\nu}{8}\|\nabla v\|^2+2C_1^2\nu^{-1}\|\zeta\|_{L^4}^4, 
\label{3.43}\\
\bigl|(\langle v,\nabla\rangle\zeta,v)\bigr|
&\le C_1\|\zeta\|_1\|v\|_{L^4}^2\le C_2\|\zeta\|_1\|v\|\,\|\nabla v\|
\notag\\
&\le \frac{\nu}{8}\|\nabla v\|^2+2C_2^2\nu^{-1}\|\zeta\|_1^2\|v\|^2. 
\label{3.44}
\end{align}
Substituting these inequalities into~\eqref{3.7}, we derive
$$
\p_t\|v\|^2+\nu\|\nabla v\|^2
\le C_3(\nu)\bigl(\|h\|^2+\|\zeta\|_{L^4}^4+\|\zeta\|_1^2\|v\|^2\bigr),
$$
whence, by Gronwall's inequality, we obtain
\begin{equation} \label{3.8}
\sup_{t\in J}\Bigl(\|v(t)\|^2+\int_0^t\|v(s)\|_1^2\dd s\Bigr)
\le C_4(\nu,M)\bigl(\|u_0\|^2+1\bigr).
\end{equation}

\smallskip
We now establish the boundedness of the norm of~$v$ in~$\XX$. To this end, we denote by $\Pi:L^2\to H$ Leray's projection to the space~$H$ of divergence-free vector fields tangent to the boundary (see~\eqref{3.30}) and take the scalar product in~$L^2$ of the first equation in~\eqref{3.5} with the function $-2\Pi\Delta v$. This results in
\begin{equation} \label{3.9}
\p_t\|\nabla v\|^2+2\nu\|\Pi\Delta v\|^2
=-2(h,\Pi\Delta v)-2\bigl(\langle \zeta+v,\nabla\rangle(\zeta+v),\Pi\Delta v\bigr).
\end{equation}
By Schwarz's inequality, 
\begin{align*}
|(h,\Pi\Delta v)|&\le  \frac{\nu}{4}\|\Pi\Delta v\|^2+C_1\|h\|^2,\\
\bigl|\bigl(\langle v+\zeta,\nabla\rangle(v+\zeta),\Pi\Delta v\bigr)\bigr|
&\le \e\|v\|_2^2+C_1\bigl\|(\zeta+v)\otimes(\zeta+v)\bigr\|_1^2,
\end{align*}
where $w\otimes w$ denotes the $2\times 2$ matrix with entries $w_iw_j$, and $\e>0$ is a small parameter. Using Sobolev's embeddings and interpolation inequalities, the boundedness of~$\zeta$ in~$C(J,H^2)$, as well as~\eqref{3.8}, we derive
\begin{align*}
\bigl\|(\zeta+v)\otimes(\zeta+v)\bigr\|_1^2
&\le \|\zeta+v\|_1^2\|\zeta+v\|\,\|\zeta+v\|_2\\
&\le \e\|v\|_2^2+C_5(\e,\nu,M)\bigl(\|u_0\|^2+1\bigr)\bigl(\|v\|_1^4+1\bigr),
\end{align*}
Recalling that the norms $\|\Pi\Delta v\|$ and~$\|v\|_2$ are equivalent and substituting the above inequalities into~\eqref{3.9}, we obtain
$$
\p_t\|\nabla v\|^2+C_6\nu\|v\|_2^2
\le C_7\bigl(\|h\|^2+(\|u_0\|^2+1)(\|v\|_1^4+1)\bigr). 
$$
Using again Gronwall's inequality and~\eqref{3.8}, we derive
\begin{equation} \label{3.10}
\sup_{t\in J}\Bigl(\|v(t)\|_1^2+\int_0^t\|v(s)\|_2^2\dd s\Bigr)
\le C_8\bigl(\nu,M,\|u_0\|_1\bigr). 
\end{equation}
Finally, applying Leray's projection~$\Pi$ to the first equation in~\eqref{3.5} and taking the~$L^2$ norm, we easily conclude that $\|\p_tv\|_{L^2(J\times D)}$ also remains bounded. We have thus proved that~$\boR:V\times\XX_2\to\XX$ is a bounded mapping. 

\smallskip
It remains to prove the continuity of~$\boR$ and inequality~\eqref{LC}. Let us take two pairs $(u_{0i},\zeta_i)$, $i=1,2$, and denote 
$$
v_i=\boR(u_{0i},\zeta_i),\quad u_i=\zeta_i+v_i, \quad v=v_1-v_2,
\quad \zeta=\zeta_1-\zeta_2.
$$ 
Then  $v\in\XX\cap C(J,V)$ is a solution of the equation
\begin{equation} \label{3.11}
\p_tv+\langle \zeta+v,\nabla\rangle u_1+\langle u_2,\nabla\rangle (\zeta+v)-\nu\Delta v+\nabla p=h:=(\nu\Delta-\p_t)\zeta. 
\end{equation}
Taking the scalar product in~$L^2$ of Eq.~\eqref{3.11} and the function $-2\Pi\Delta v$ and using some estimates similar to those above, we establish that~$\boR:V\times\XX_2\to\XX$ is Lipschitz continuous on every bounded subset. Finally, to prove~\eqref{LC}, it suffices to take the scalar product in~$L^2$ of Eq.~\eqref{3.11} with~$v$ and to carry out standard arguments. 

\smallskip
{\it Proof of~{\rm(b)}}. 
We shall need a result from the theory of the non-autonomous Stokes equations  in~$L^q$ spaces. Namely, we consider the problem 
\begin{equation} \label{2.016}
\p_t v-\nu\Delta v+\nabla p=h(t,x), \quad \diver v=0, \quad x\in D,
\end{equation}
supplemented with the initial and boundary conditions~\eqref{3.6}. Let us denote by~$e^{tL_q}$ the resolving semigroup of the homogeneous problem (corresponding to $h\equiv0$) with an initial condition~$u_0\in L^q\cap H$ and by~$L_q$ the corresponding generator, which is a closed operator in $L^q\cap H$. In view of Proposition~1.2 in~\cite{GM-1985} (see also Theorem~2 in~\cite{giga-1981}), the operator~$e^{tL_q}$ is continuous from~$L^q\cap H$ to the domain~$\DD(L_q^\alpha)$ of the operator~$L_q^\alpha$ for any $\alpha\ge0$ and $t>0$, and
\begin{equation} \label{2.017}
\bigl\|e^{tL_q}\bigr\|_{\LL(L^q\cap H,\DD(L_q^\alpha))}\le C_{q,\alpha}t^{-\alpha}. 
\end{equation}
In view of Duhamel's formula, the solution~$v(t,x)$ for problem~\eqref{2.016}, \eqref{3.6} with $u_0\in V$ and $h\in L^s(J,L^q)$ can be written as
\begin{equation} \label{2.018}
v(t)=e^{tL_q}u_0+\int_0^te^{(t-\theta)L_q}(\Pi h)(\theta)\,\dd\theta.
\end{equation}
Since the projection $\Pi:L^2\to H$ is continuous from~$L^q$ to~$L^q\cap H$ for any $q\in(2,\infty)$, and~$\DD(L_q^\alpha)$ is continuously embedded into~$W^{2\alpha,q}(D)$, it follows from~\eqref{2.017} and~\eqref{2.018} that, for any $s>2$, we have
\begin{align}
\|v(t)\|_{W^{1,q}}&\le C_9t^{-1/2}\|u_0\|_{L^q}
+C_9\int_0^t (t-\theta)^{-1/2}\|h(\theta)\|_{L^q}\dd\theta\notag\\ 
&\le C_{10}\bigl(t^{-1/2}\|u_0\|_V+\|h\|_{L^s(J,L^q)}\bigr), \quad t\in J, 
\label{2.019}
\end{align}
where we used H\"older's inequality and the continuity of the embedding $V\subset L^q$ for $1\le q<\infty$. It follows, in particular, that~$v$ is a continuous function on the interval~$(0,1]$ with range in~$W^{1,q}$. 

On the other hand, if $h\in L^2(J,H^r)$ for some $r\in(0,\frac12)$, then for any $u_0\in V$ problem~\eqref{2.016}, \eqref{3.6} has a unique solution $v\in\XX$, which belongs to $L^2(J_\tau,H^{r+2}\cap V)\cap W^{1,2}(J_\tau,H^r)$ for any $\tau\in(0,1)$. By interpolation, this space is embedded into $C(J_\tau, H^{r+1})$, which is a subspace of~$C(J_\tau,W^{1,q})$ with $q=\frac{2}{1-r}$. Moreover, we have an analogue of inequality~\eqref{2.019}:
\begin{equation}
	\|v(t)\|_{W^{1,q}}
\le C_{11}\bigl(t^{-1/2}\|u_0\|_V+\|h\|_{L^2(J,H^r)}\bigr), \quad t\in J. 
\label{2.0190}
\end{equation}

We now go back to the regularity of $\SSS^\tau(u_0,\eta)$. Since $u=\zeta+v$, where $\zeta\in\XX_2\subset C(J,H^2)$, the required properties will be established if we prove that they hold for the solution $v\in\XX_1$ of problem~\eqref{3.5}, \eqref{3.6}. Let us rewrite~\eqref{3.5} in the form~\eqref{2.016}, where 
\begin{equation} \label{2.20}
h(t,x)=h_1+h_2, \quad h_1:=(\nu\Delta-\p_t)\zeta, \quad h_2:=-\langle u,\nabla\rangle u. 
\end{equation}
We claim that $h_1\in L^2(J,H^1)$, $h_2\in L^s(J,L^q)$ for any $q<+\infty$ and some $s=s_q>2$, and 
\begin{equation} \label{2.21}
\|h_1\|_{L^2(J,H^1)}+\|h_2\|_{L^s(J,L^q)}\le C_{12}\bigl(\|\eta\|_G+\|u\|_{\XX_1}^2\bigr).
\end{equation}
In view of~\eqref{2.019} and~\eqref{2.0190}, this will imply all the required properties. 

Since $\zeta\in\XX_2$, the function~$h_1$ belongs to the space $L^2(J,H^1)$, and its norm is bounded by~$\|\eta\|_G$. Furthermore, since  $u\in \XX_1$, we have 
$$
u\in C(J,H^1), \quad \nabla\otimes u\in C(J,L^2)\cap L^2(J,H^1),
$$
and the corresponding norms are bounded by~$\|u\|_{\XX_1}$. Using the interpolation inequality $\|w\|_{L^p}\le C_{13}\|w\|^{2/p}\|w\|_1^{1-2/p}$ and the continuous embedding $H^1\subset L^p$, we derive
$$
\|\langle u,\nabla\rangle u\|_{L^q}\le \|u\|_{L^{\lambda q}}\|\nabla\otimes u\|_{L^{\lambda q/(\lambda-1)}}
\le C_{14}\|u\|_1^{1+r_\lambda}\|u\|_2^{1-r_\lambda}, 
$$
where $\lambda\in (1,\infty)$ is arbitrary, $r_\lambda=\frac{2(\lambda-1)}{\lambda q}$, and~$C_{14}$ depends only on~$\lambda$ and~$q$. Given any $q\in(2,\infty)$, we choose $s>2$ such that $q<\frac{2s}{s-2}$ and set $\lambda=\frac{2s}{2s-qs+2q}$, so that $(1-r_\lambda)s=2$. In this case, 
\begin{equation*}
\|h_2\|_{L^s(J,L^q)}\le C_{15} \|u\|_{C(J,H^1)}^{1+r_\lambda}\|u\|_{L^2(J,H^2)}^{1-r_\lambda}\le C_{16}\|u\|_{\XX_1}^2. 
\end{equation*}
This completes the proof~\eqref{2.21} and that of the proposition.

\section{Exponential mixing for the Navier--Stokes system with boundary noise}
\label{s5}

In this section, we apply Theorem~\ref{t1.2} to the 2D Navier--Stokes system driven by a boundary noise. We first formulate the main result and outline the key steps of the proof. The details are given in Sections~\ref{s3.2} and~\ref{proof-p26}. 

\subsection{Main result}
\label{s3.1}
Let us consider problem~\eqref{0.4}, \eqref{0.5}, in which~$\eta$ is a random process of the form~\eqref{0.6}. It is assumed that~$\{\eta_k\}$ entering~\eqref{0.6} is a sequence of i.i.d.\ random variables in~$G$ such that $\eta_k(k-1)=0$ almost surely for any~$k\ge1$. It follows from Proposition~\ref{p3.1} that, for any $V$-valued random variable~$u_0$, there is a unique random process~$u(t,x)$ whose almost every trajectory satisfies the inclusions
$$
u\in L_{\mathrm{loc}}^2(\R_+,H_\sigma^2), \quad 
\p_tu\in L_{\mathrm{loc}}^2(\R_+,L_\sigma^2) 
$$
and Eqs.~\eqref{0.4}, \eqref{0.5}, and~\eqref{3.1}. To formulate our main result, we define the outside lateral boundary $\widetilde\Sigma=(0,1)\times\p\widetilde D$ and introduce the following condition concerning the law~$\ell$ of the random variables~$\eta_k$. 

\medskip
{\bf Structure of the noise.}\ 
{\sl There is an open subset $\Sigma_0\subset\widetilde\Sigma$ whose closure~$\overline{\Sigma}_0$ is compact  in $\widetilde\Sigma$ such that the support of~$\ell$ is contained in the vector space $G(\Sigma_0):=\{v\in G:\supp v\subset\overline{\Sigma}_0\}$. Moreover, there exists an orthonormal basis~$\{\varphi_j\}$ in~$G(\Sigma_0)$, a sequence of non-negative numbers $\{b_j\}$, and independent scalar random variables $\xi_{jk}$ with values in~$[-1,1]$ such that 
\begin{equation} \label{3.31}
\eta_k(t,x)=\sum_{j=1}^\infty b_j\xi_{jk}\varphi_j(t,x), 
\quad \BBBB:=\sum_{j=1}^\infty b_j^2<\infty. 
\end{equation}
Finally, there are non-negative functions $p_j\in C^1(\R)$ such that 
\begin{equation} \label{3.32}
p_j(0)\ne0, \quad \DD(\xi_{jk})=p_j(r)\,\dd r\quad\mbox{for any $j\ge1$}.
\end{equation}
}

\noindent
This hypothesis implies that the random perturbation~$\eta$ is space-time localised  in~$\Sigma_0$ (so that the perturbation acts only through the boundary\,\footnote{Our result remains true in the more general setting when the random perturbation may be non-zero on the boundaries of the interior domains~$D_i$, $i=1,\dots,m$. In this case, however, one should add the condition that the circulation (i.e., the integral of the normal velocity) is zero on the boundary of each of the domains~$D_i$; cf.\ Proposition~\ref{p4.8}.}~$\p\widetilde D$) and possesses some regularity properties. The following theorem, which is the main result of this paper, shows that if the law of~$\eta_k$ is sufficiently non-degenerate, then the corresponding random flow is exponential mixing. Recall that the space~$\XX$ is defined by~\eqref{XX}. 

\begin{theorem} \label{t3.1}
Let the above hypotheses be satisfied and let $\BBBB_0>0$ be any fixed number. In this case, for any $\nu>0$ there is an integer $N_\nu\ge1$ such that, if
\begin{equation} \label{ND}
\BBBB\le \BBBB_0, \qquad  b_j\ne0\quad\mbox{for $j=1,\dots,N_\nu$},
\end{equation}
then the following property holds: there is a measure $\mmu_\nu\in\PP(\XX)$ and positive numbers~$C_\nu$ and~$\gamma_\nu$ such that, for any $u_0\in V$, the solution~$u(t,x)$ of~\eqref{0.4}, \eqref{0.5}, \eqref{3.1} satisfies the inequality
\begin{equation} \label{3.33}
\|\DD(\uuu_k)-\mmu_\nu\|_L^*
\le C_\nu e^{-\gamma_\nu k} \quad\mbox{for $k\ge C_\nu\log(1+\|u_0\|_1)$},
\end{equation}
where $\uuu_k$ stands for the restriction of~$u(t+k-1)$ to~$[0,1]$, and the dual-Lipschitz norm~$\|\cdot\|_L^*$ is taken over the space~$\XX$. Moreover, for any $V$-valued random variable~$u_0$  independent of~$\eta$, we have
\begin{equation} \label{3.13}
\|\DD(u(t))-\mu_\nu(\bar t)\|_L^*
\le C_\nu e^{-\gamma_\nu t}\bigl(1+\E\,\|u_0\|_1\bigr), \quad t\ge0,
\end{equation}
where $\bar t\in[0,1)$ stands for the fractional part of~$t\ge0$, $\mu_\nu(s)\in\PP(H_\sigma^1)$ denotes the projection of~$\mmu_\nu$ to the time $t=s$,  and the dual-Lipschitz norm~$\|\cdot\|_L^*$ is taken over the space~$H_\sigma^1$. 
\end{theorem}

Let us note that if $b_j\ne0$ for all $j\ge1$, then the result is true for any $\nu>0$. 
We also remark that the $H^1$-regularity of the initial condition~$u_0$ is not really needed: we can take any $H$-valued function~$u_0$ independent of~$\eta$ (see~\eqref{3.30} for the definition of~$H$), and the regularisation property of the Navier--Stokes flow will ensure that $u(t)\in H_\sigma^1$ almost surely for any $t>0$. 

A detailed proof of Theorem~\ref{t3.1} is given in the next two subsections. Here we outline briefly the main idea. 

The dissipativity of the 2D Navier--Stokes system driven by a circulation-free boundary forcing enables one to prove that any solution of~\eqref{0.4}--\eqref{0.6} satisfies the inequality
\begin{equation} \label{3.34}
\|u(t)\|_1\le C_1(e^{-\alpha t}\|u_0\|_1+1), \quad t\ge0,
\end{equation}
where $C_1$ and~$\alpha$ are positive numbers depending only on~$\nu$. It follows that the stochastic flow restricted to integer times possesses a compact invariant absorbing set $X\subset V$. Furthermore, since~$\{\eta_k\}$ is a sequence of i.i.d.\ random variables in~$G(\Sigma_0)$, the family of all trajectories issued from~$X$ and restricted to integer times form a Markov process~$(u_k,\IP_v)$. The key point of the proof is the verification of the hypotheses of Theorem~\ref{t1.2} for~$(u_k,\IP_v)$, from which we conclude that inequality~\eqref{4.12} holds for the corresponding Markov semigroup. Combining this with a result on the behaviour of the dual-Lipschitz metric under a Lipschitz mapping, we arrive at~\eqref{3.33}. Finally, inequality~\eqref{3.13} is a simple consequence of~\eqref{3.33}. 

\subsection{Proof of Theorem~\ref{t3.1}}
\label{s3.2}

{\it Step~1: Compact absorbing invariant set\/}. 
We claim that the random flow generated by~\eqref{0.4}, \eqref{0.5} possesses a compact invariant absorbing set. More precisely, there is a compact set $X\subset V$ such that
\begin{align}
&\IP\{S_k(u_0,\eta)\in X\mbox{ for any }k\ge0\}=1&\quad
&\mbox{for any $u_0\in X$},\label{3.35}\\
&\IP\{S_k(u_0,\eta)\in X\mbox{ for }k\ge C_2\log(\|u_0\|_1+3)\}=1&
\quad&\mbox{for any $u_0\in V$},
\label{3.36}
\end{align}
where $C_2\ge1$ does not depend on~$u_0$. To this end, it suffices to establish~\eqref{3.34}. Indeed, if~\eqref{3.34} is proved, then we have
\begin{equation} \label{3.37}
\|S_t(u_0,\eta)\|_1\le R\quad\mbox{for $t\ge T(\|u_0\|_1)$},
\end{equation}
where $R=2C_1$ and $T(r)=\alpha^{-1}\log(r+1)$. It follows from~\eqref{3.31} that the support~$\KK$ of the law of~$\eta_k$ is a compact subset of~$G$ that is included in $G(\tau_0)$ for some~$\tau_0>0$ (see~\eqref{3.38}). Let us denote by~$k_0(R)\ge1$ the least integer larger than $\alpha^{-1}\log(2R+1)$ and define
\begin{equation} \label{3.038}
X=\bigcup_{k=1}^{k_0(R)}\aA_k(R,\KK),
\end{equation}
where the sets~$\aA_k$ are defined recursively by the relations
$$
\aA_1(R,\KK)=S(B_V(R),\KK), \quad \aA_k(R,\KK)=S(\aA_{k-1}(R,\KK),\KK)
\quad\mbox{for $k\ge2$},
$$
and $S=S_1$. The regularising property of the flow for the homogeneous Navier--Stokes system implies that each of the sets~$\aA_k(R,\KK)$ is compact, and therefore so is their finite union~$X$. Relations~\eqref{3.35} and~\eqref{3.36} follow immediately from~\eqref{3.37} and the definition of~$X$. 

\smallskip
To prove~\eqref{3.34}, we first establish an estimate for the~$L^2$ norm of solutions. Namely, we claim that 
\begin{equation} \label{3.39}
\|S_t(u_0,\eta)\|\le C_3(e^{-\alpha t}\|u_0\|+1), \quad t\ge0,
\end{equation}
where $C_3>0$ does not depend on~$u_0\in V$. Indeed, let us fix $\e>0$ and denote by $Q_\e:G(\Sigma_0)\to\XX_2$ the continuous linear operator constructed in Proposition~\ref{p4.8}. We now define a random process~$\zeta_\e$ by the relation
\begin{equation} \label{3.41}
\zeta_\e(t)=(Q_\e\eta_k)(t-k+1)\quad\mbox{for $t\in[k-1,k]$ and $k\ge1$}. 
\end{equation}
It follows from~\eqref{4.62} that 
\begin{equation} \label{3.40}
\bigl|\bigl(\langle v,\nabla\rangle \zeta_\e (t),v\bigr)_{L^2}\bigr|
\le C_4\e\,\|v\|_1^2\quad\mbox{for any $v\in V$}, 
\end{equation}
where $C_4 =\sup_{t\in J}\|\eta(t)\|_{3/2}<\infty$. Let us represent a solution~$u=S_t(u_0,\eta)$ of~\eqref{0.4}--\eqref{0.6} in the form $u=\zeta_\e+v$. Then~$v$ must be a solution of problem~\eqref{3.5}, \eqref{3.6}, in which $\zeta=\zeta_\e$. Taking the scalar product in~$L^2$ of the first equation in~\eqref{3.5} with the function~$2v$, we obtain Eq.~\eqref{3.7} in which $\zeta=\zeta_\e$. Using~\eqref{3.42}, \eqref{3.43}, and~\eqref{3.40} and choosing $\e>0$ sufficiently small, we derive 
$$
\p_t\|v\|^2+\nu\|\nabla v\|^2
\le C_5\bigl(\|h\|^2+\|\zeta_\e\|_{L^4}^4\bigr). 
$$
Application of Gronwall's inequality completes the proof of~\eqref{3.39}. 

We now prove~\eqref{3.34}. Since~$\zeta_\e(t)$ is bounded in~$H^1$, it suffices to establish inequality~\eqref{3.34} with $u=v$. Its validity for $0\le t\le1$ follows immediately from~\eqref{3.10}. Assuming now that $t\ge1$, we write $v(t)=R_1(v(t-1),\zeta_\e^t)$, where $R_t:H\times\XX_2\to H$ denotes the resolving operator for~\eqref{3.5}, \eqref{3.6}, and $\zeta_\e^t$ stands for the   function $s\mapsto\zeta_\e(s-t+1)$. Combining this with the regularising property for~$R_1$ (e.g., Theorem~6.2 in~\cite[Chapter~1]{BV1992}) and the boundedness of the norm of the function~$\zeta_\e^t$ in the space~$\XX_2$, we see that 
\begin{equation} \label{3.45}
\|v(t)\|_1\le C_6\bigl(\|v(t-1)\|+1).
\end{equation}
On the other hand, it follows from~\eqref{3.39} and the boundedness of the $L^2$ norm of~$\zeta_\e(t)$ that 
$$
\|v(t)\|\le C_7(e^{-\alpha t}\|u_0\|+1)\quad\mbox{for all $t\ge0$}. 
$$
Combining this with~\eqref{3.45}, we arrive at~\eqref{3.34}. 

\medskip
{\it Step~2: Reduction to the dynamics at integer times\/}. 
In view of~\eqref{3.35}, we can consider the discrete-time Markov process~$(u_k,\IP_v)$ defined by~\eqref{1.1} in the phase space~$X$. Suppose we have shown that~$(u_k,\IP_v)$ has a unique stationary measure $\mu_\nu\in\PP(X)$, which is exponentially mixing in the dual-Lipschitz metric over the space~$X$, so that we have inequality~\eqref{4.12}, in which $\mu=\mu_\nu$ and~$\PPPP_k^*$ denotes the Markov semigroup associated with~$(u_k,\IP_v)$. Let us denote by~$\mmu_\nu\in\PP(\XX)$ the image of the product measure $\mu_\nu\otimes\ell\in\PP(X\times G)$ under the mapping $(u,\eta)\mapsto \SSS(u,\eta)$. 
We claim that both~\eqref{3.33} and~\eqref{3.13} hold. To prove this, we shall use the following lemma, whose proof follows immediately from the definition of the dual-Lipschitz distance.

\begin{lemma} \label{l2.4}
\begin{itemize}
\item[\bf(i)]
Let $X_1$ and~$X_2$ be Polish spaces and let $F:X_1\to X_2$ be a $C$-Lipschitz mapping. Then, for any $\mu,\mu'\in\PP(X_1)$, we have
\begin{equation} \label{3.14}
\|F_*(\mu)-F_*(\mu')\|_L^*\le C\|\mu-\mu'\|_L^*,
\end{equation}
where the dual-Lipschitz metrics on the left- and right-hand sides are taken over the spaces~$X_2$ and~$X_1$, respectively.  
\item[\bf(ii)]
Let $X$ and~$G$ be Polish spaces and let $\mu,\mu'\in\PP(X)$ and~$\lambda\in\PP(G)$ be some measures. Then
\begin{equation} \label{3.15}
\|\mu\otimes\lambda-\mu'\otimes\lambda\|_L^*=\|\mu-\mu'\|_L^*. 
\end{equation}
\end{itemize}
\end{lemma}

To prove~\eqref{3.33}, let us fix $u_0\in V$. In view of~\eqref{3.36}, there is an integer~$T_0\ge1$ of order $\log\|u_0\|_1$ such that $\IP\{S_{T_0}(u_0,\eta)\in X\}=1$. Therefore, by the Markov property, we can assume from the very beginning that $u_0\in V$ and establish~\eqref{3.33} for all $k\ge0$.  

Inequality~\eqref{4.12} implies that 
\begin{equation} \label{3.16}
\|\DD(u(k))-\mu_\nu\|_L^*\le C_8e^{-\gamma_\nu k},\quad k\ge0,
\end{equation}
where $C_8$ and~$\gamma_\nu$ are some positive numbers, and the dual-Lipschitz norm is taken over the space~$V$. Now note that $\DD(\uuu_k)$ is the image of the product measure $\DD(u(k-1))\otimes\ell$ under the mapping~$\SSS$. Therefore, combining~\eqref{3.16} with~\eqref{3.14} and~\eqref{3.15}, we arrive at~\eqref{3.33}. 

To prove~\eqref{3.13}, we first note that it suffices to consider the case of a deterministic initial condition. Furthermore, since~$\XX$ is continuously embedded into~$C(J,H_\sigma^1)$, the linear application $v\mapsto v(s)$ is continuous from~$\XX$ to~$H_\sigma^1$. Hence, it follows from~\eqref{3.33} and assertion~(i) of Lemma~\ref{l2.4} that inequality~\eqref{3.13} with a deterministic~$u_0\in V$ holds for $t\ge C_\nu\log(1+\|u_0\|_1)$.  Its validity (with a sufficiently small $\gamma_\nu>0$)  for $t\le  C_\nu\log(1+\|u_0\|_1)$ follows from~\eqref{3.34}.   

\smallskip
Thus, to prove Theorem~\ref{t3.1}, it suffices to show that the hypotheses of Theorem~\ref{t1.2} are satisfied for the discrete-time Markov process~$(u_k,\IP_v)$ with the phase space~$X$.

\medskip
{\it Step~3: Reduction to controllability\/}.
We apply Theorem~\ref{t1.2} in which $H=V$, $E=G(\Sigma_0)$, $S(u,\eta)$ is the time-$1$ resolving operator for problem~\eqref{0.4}, \eqref{0.5}, $X$~is given by~\eqref{3.038}, and~$\KK$ is the support of the law~$\ell$ of the random variables~$\eta_k$.

The hypotheses imposed on~$\ell$ in Theorem~\ref{t1.2} are obviously satisfied (see the description of the structure of~$\eta_k$ in Section~\ref{s3.1}). We thus need to check the  conditions on~$S$. Namely, we shall prove that the global approximate controllability to some point $\hat u\in V$ and local stabilisability are true. 

The global approximate controllability to the point $\hat u=0$ is an easy consequence of the dissipativity of the homogeneous Navier--Stokes problem. Indeed, the solution of problem~\eqref{0.4}, \eqref{0.5} with $\eta\equiv0$ satisfies the inequality 
$$
\|u(t)\|\le e^{-\alpha t}\|u(0)\|\quad\mbox{for $t\ge0$},
$$
where $\alpha>0$ does not depend on~$u$. Combining this with the regularising property of~$S_t(u_0,\eta)$ (e.g., see Theorem~6.2 in~\cite[Chapter~I]{BV1992}), we see that
\begin{equation} \label{2.037}
\|S_k(v,0)\|_1\le C_9e^{-\alpha k}\quad\mbox{for all $k\ge0$, $v\in X$},
\end{equation}
where $C_9>0$ does not depend on~$v$ and~$k$. Since $0\in\KK$, we conclude from~\eqref{2.037} that the global approximate controllability to~$\hat u=0$ is true. 

We now turn to the more complicated property of local stabilisability. To prove it, we shall apply a well-known idea in the control theory of PDEs: we extend the domain through the controlled part of the boundary, establish the required property by a distributed control with support in the extended part, and then define the control for the initial problem by restricting the constructed solution to the boundary; see Chapter~III in~\cite{FI1996}. We describe here the main ideas (omitting some unessential technical details), and give a complete proof in Steps~4 and~5. 

\smallskip
We wish to prove that, given sufficiently close initial conditions $u_0,u_0'\in X$ and a boundary function $\eta\in B_{G(\Sigma_0)}(R)$, one can find $\eta'\in G(\Sigma_0)$ of the form $\eta'=\eta+\varPhi(u_0,u_0',\eta)$ such that the mapping~$\varPhi$ is continuous in $(u_0,u_0',\eta)$ and continuously differentiable in~$\eta$, its image is contained in a finite-dimensional subspace $\EE\subset G(\Sigma_0)$, and we have the inequalities
\begin{align}
\|\varPhi(u_0,u_0',\eta)\|_G+\|D_\eta\varPhi(u_0,u_0',\eta)\|_{\LL(G)}
&\le C\|u_0-u_0\|_1,\label{2.41}\\
\|S(u_0,\eta)-S(u_0',\eta+\varPhi(u_0,u_0',\eta))\|_1&\le q\|u_0-u_0'\|_1,
\label{2.42}
\end{align}
where $C$ and~$q<1$ are positive numbers not depending on $(u_0,u_0',\eta)$. Let us choose numbers $0<a<b<c<1$ and a connected segment~$\Gamma$ of the external boundary of~$D$ such that $[a,c]\times\Gamma\subset\Sigma_0$. Setting $\xi=\varPhi(u_0,u_0',\eta)$, we define $\xi(t)=0$ for $t\le a$. To construct~$\xi$ on~$[a,1]$, we set $u_a=S_a(u_0,\eta)$ and $u_a'=S_a(u_0',\eta)$, and seek a solution of the form $u'=u+w$. Then~$w$ must satisfy the equations 
\begin{gather}
\p_tw+\langle u,\nabla\rangle w +\langle w,\nabla\rangle u+\langle w,\nabla\rangle w
-\nu\Delta w+\nabla p=0, \quad \diver w=0,\label{2.43}\\
w\bigr|_{\p D}=\xi, \quad w(a)=w_a:=u_a'-u_a.\label{2.44}
\end{gather}
Note that $w_a\in V$. Suppose we found~$\xi$, with appropriate regularity and bound on its norm, such that 
\begin{equation} \label{2.45}
\|w(b)\|_1\le \e\|w_a\|_1,
\end{equation}
where $\e>0$ is sufficiently small. We then extend~$\xi$ to~$[b,1]$ so that its norm is still controlled and $\xi(t)=0$ for $t\ge c$. All required properties are then derived from the above description.

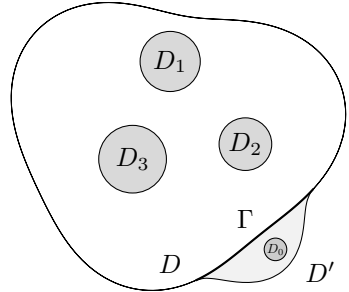
\begin{wrapfigure}{r}{0.35\textwidth}
\vspace{-10pt}
\begin{center}
\begin{tikzpicture}
\filldraw[draw=black,fill=gray!10]
plot [domain=0:360,smooth,samples=100] (\x:{2+sin(3*\x)/4+cos(2*\x)/6+cos(3*\x+25)/7})
-- plot [domain=360:0,smooth,samples=100] (\x:{2+sin(3*\x)/4+cos(2*\x)/6+cos(3*\x+25)/7+70*max(0,(\x/100-2.8)*(3.4-\x/100))*abs(\x/100-2.8)*abs(3.4-\x/100))})
-- cycle;
\draw[domain=280:340,thick,smooth,samples=100] 
plot (\x:{2+sin(3*\x)/4+cos(2*\x)/6+cos(3*\x+25)/7});
\draw[fill=gray!30!white] (0,1) circle(0.4) node {$D_1$}; 
\draw[fill=gray!30!white] (1,-0.1) circle(0.35) node {$D_2$}; 
\draw[fill=gray!30!white] (-0.5,-0.3) circle(0.45) node {$D_3$}; 
\draw[fill=gray!40!white] (1.4,-1.5) circle(0.15) node {$\scalebox{0.5}{$D_0$}$}; 
\draw (0,-1.7) node {$D$}; 
\draw (2,-1.8) node {$D'$}; 
\draw (1,-1.1) node {$\Gamma$}; 
\end{tikzpicture}
\end{center}
\captionof{figure}{The domain~$D'$\label{pic2}}
\end{wrapfigure}
The key point is the proof of~\eqref{2.45}. To this end, we construct a one-connected domain~$\widetilde D'\supset\widetilde D$ with smooth boundary~$\p\widetilde D'$ such that 
$$
\p\widetilde D\setminus(\p\widetilde D'\cap \p\widetilde D)=\Gamma,
$$ 
and define (see Figure~\ref{pic2} and cf.~\eqref{3.2})
$$
D'=\widetilde D'\setminus\bigcup_{i=1}^m{\overline D}_i.
$$
We next use Corollary~\ref{c4.10} to extend the functions~$u$ to the domain $[a,1]\times D'$ and also extend~$w_a$ to~$D'$ by zero. Denote the extended functions by~$\tilde u$ and~$\widetilde w_a$, respectively, and remark that $\widetilde w_a$ belongs to the space~$V$ considered on~$D'$. Let us fix an open set $D_0\subset\R^2$ such that $\overline D_0\subset D'\setminus\overline D$ and consider the following  problem with distributed control:
\begin{gather}
\p_t{\widetilde w}+\langle \tilde u,\nabla\rangle {\widetilde w} +\langle {\widetilde w},\nabla\rangle \tilde u+\langle {\widetilde w},\nabla\rangle {\widetilde w}
-\nu\Delta {\widetilde w}+\nabla p=f, 
\quad \diver\widetilde w = 0,\label{2.46}\\
{\widetilde w}\bigr|_{\p D'}=0, \quad {\widetilde w}(a)=\widetilde w_a,\label{2.47}
\end{gather}
where $f(t,x)$ is a control function supported by~$[a,b]\times D_0$. We shall construct~$f$ such that the solution~$\widetilde w$ of~\eqref{2.46}, \eqref{2.47} satisfies inequality~\eqref{2.45} in which~$w$ is replaced by~$\widetilde w$. In this case, the restriction of~$\widetilde w$ to~$[a,b]\times D$ will be a solution of~\eqref{2.43}, \eqref{2.44} with $\xi=\widetilde w|_{\p D}$ and will satisfy~\eqref{2.45}. Let us mention that, in the proof below, we shall need to replace the function~$\widetilde w_a$ in~\eqref{2.47} by its regularisation (in order to have $\xi\in G(\Sigma_0)$), to establish a stronger version of~\eqref{2.45}, to follow the dependence of the control~$\xi$ on the data, and to ensure that it belongs to a finite-dimensional subspaces of~$G(\Sigma_0)$. 

\medskip
{\it Step~4: Construction of a control\/}. Given $\delta>0$, we set
$$
\DDD_\delta=\{(u_0,u_0')\in X\times X:\|u_0-u_0'\|_1\le\delta\}.
$$ 
We need to construct, for any $R>0$ and a sufficiently small $\delta>0$, a continuous mapping $\varPhi:\DDD_\delta\times B_{G(\Sigma_0)}(R)\to G(\Sigma_0)$, $(u_0,u_0',\eta)\mapsto \eta'$, that is continuously differentiable in~$\eta$, has an image contained in a finite-dimensional subspace~$\EE$, and satisfies~\eqref{2.41} and~\eqref{2.42}.  We begin with a simple reduction. 

Recall that the space~$G_s$ with $3/2\le s\le 2$ was defined before Proposition~\ref{p3.1}. We claim that it suffices to construct a Banach space $F\subset G$, compactly embedded into $G_s$ for some $s\in(3/2,2)$ and, for any given $\varkappa>0$, a continuous mapping 
$$
\varPhi':\DDD_\delta\times B_{G(\Sigma_0)}(R)\to G(\Sigma_0)\cap F
$$ 
such that $\varPhi(u_0,u_0',\eta)$ is continuously differentiable in~$\eta$, and 
\begin{align}
\|\varPhi'(u_0,u_0',\eta)\|_F+\|D_\eta\varPhi'(u_0,u_0',\eta)\|_{\LL(G)}
&\le C\|u_0-u_0\|_1,\label{2.48}\\
\|S_\tau(u_0,\eta)-S_\tau(u_0',\eta+\varPhi'(u_0,u_0',\eta))\|
&\le \varkappa\,\|u_0-u_0'\|_1,
\label{2.49}
\end{align}
where $\tau\in(0,1)$ is a fixed number such that $\Sigma_0\subset[0,\tau]\times\p\widetilde D$, and~$C>0$ may depend on~$\varkappa$. Indeed, if such a mapping is constructed, then denoting by~${\mathsf P}_N$ the orthogonal projection in~$G(\Sigma_0)$ onto the vector span of\,\footnote{Recall that~$\{\varphi_j\}$ is the orthonormal basis in~$G(\Sigma_0)$ entering~\eqref{3.31}.} $\{\varphi_1,\dots,\varphi_N\}$, we define $\varPhi={\mathsf P}_N\circ\varPhi'$. Let us prove that if~$\varkappa$ and~$N^{-1}$ are sufficiently small, then~$\varPhi$ satisfies all the required properties. 

The image of~$\varPhi$ is contained in the $N$-dimensional subspace~$\EE$  spanned by the first~$N$ vectors of the basis~$\{\varphi_j\}$. The continuity of~$\varPhi$ with respect to its arguments and its continuous differentiability in~$\eta$ are obvious, and~\eqref{2.41} is a consequence of~\eqref{2.48}. To prove~\eqref{2.42}, we first use the Lipschitz-continuity of~$S_\tau$ to write (see assertion~(a) of Proposition~\ref{p3.1})
\begin{align}
&\|S_\tau(u_0,\eta)-S_\tau(u_0',\eta+\varPhi(u_0,u_0',\eta))\| 
\le \|S_\tau(u_0,\eta)-S_\tau(u_0',\eta+\varPhi'(u_0,u_0',\eta))\|
\notag\\
&\qquad+ \|S_\tau(u_0',\eta+\varPhi'(u_0,u_0',\eta))-S_\tau(u_0',\eta+\varPhi(u_0,u_0',\eta))\|
\notag\\
&\quad\le \varkappa\|u_0-u_0'\|_1
+C_{10}\|(I-\mathsf P_N)\varPhi'(u_0,u_0',\eta)\|_{G_s}. 
\label{2.50}
\end{align}
Since the embedding $F\subset G_s$ is compact, there is a sequence~$\{\alpha_N\}$ going to zero such that 
$$
\|(I-\mathsf P_N)\zeta\|_{G_s}\le \alpha_N\|\zeta\|_F
\quad\mbox{for any $\zeta\in F$}. 
$$
Combining this with~\eqref{2.50} and~\eqref{2.48}, we see that 
$$
\|S_\tau(u_0,\eta)-S_\tau(u_0',\eta+\varPhi(u_0,u_0',\eta))\|
\le (\varkappa+C_{11} \alpha_N)\|u_0-u_0'\|_1,
$$
where $C_{11}=C_{10}C$. The functions~$\eta$ and $\eta+\varPhi(u_0,u_0',\eta)$ vanish for $t\ge \tau$, and the regularising property of the Navier--Stokes equations with no-slip boundary condition (e.g., see Theorem~6.2 in~\cite[Chapter~1]{BV1992}) implies that 
\begin{align*}
\|S(u_0,\eta)-S(u_0',\eta+\varPhi(u_0,u_0',\eta))\|_1
&\le C_\tau \|S_\tau (u_0,\eta)-S_\tau(u_0',\eta+\varPhi(u_0,u_0',\eta))\|\\
&\le C_\tau (\varkappa+C_{11}\alpha_N)\|u_0-u_0'\|_1. 
\end{align*}
Choosing~$\varkappa$ to be sufficiently small and~$N$ sufficiently large, we arrive at the required inequality~\eqref{2.42}.

\smallskip
We now apply the scheme described in Step~3 to construct a mapping~$\varPhi'$ with the above-mentioned properties. To this end, we fix numbers $0<a<b<c<\tau$ such that $[a,c]\times\Gamma\subset\Sigma_0$, and consider a pair of initial conditions $(u_0,u_0')\in\DDD_\delta$ and a boundary function $\eta\in B_{G(\Sigma_0)}(R)$. The required control $\xi=\varPhi'(u_0,u_0',\eta)$ is defined consecutively on the intervals $[0,a]$ and $[a,1]$. Let us set 
\begin{equation} \label{2.61}
  \xi(t)=0\quad\mbox{for $0\le t\le a$}. 
\end{equation}
By Proposition~\ref{p3.1}, the function $w(t)=S_t(u_0',\eta)-S_t(u_0,\eta)$ belongs to the space~$\XX_1$ on the interval $[0,a]$ and satisfies Eqs.~\eqref{2.43} and the boundary and initial conditions
\begin{equation} \label{2.62}
w\bigr|_{\p D}=0, \quad w(0)=w_0:=u_0'-u_0.   
\end{equation}
The Lipschitz continuity of the resolving operator for the Navier--Stokes-type system~\eqref{2.43} implies that
\begin{equation} \label{2.63}
  \|w(a)\|_1\le C_{12}\|w_0\|_1\le C_{12}\delta,
\end{equation}
where $C_{12}>0$ depends only on~$R$. 

To define~$\xi$ on~$[a,1]$, we use Corollary~\ref{c4.10} to extend the function $u=\SSS(u_0,\eta)$ to a larger domain~$\DD$ containing~$\overline D$. In view of part~(b) of Proposition~\ref{p3.1} and the continuity of the embedding $W^{1,q}(\DD)\subset C(\overline{\DD})$ for $q>2$, the restriction of the resulting function~$\tilde u$ to the time interval~$I:=[a,1]$ belongs to the space 
$$
\UU:=L^2(I,H_\sigma^2(\DD))\cap W^{1,2}(I,L_\sigma^2(\DD))\cap C(I\times\overline\DD). 
$$
We extend~$w_a=w(a)\in V$ to $D'\setminus D$ by zero and denote $\widetilde w_a=\Omega_\gamma w_a$, where~$\{\Omega_\gamma\}$ is the family of regularising operators constructed in Proposition~\ref{p3.12}, and $\gamma\in(0,1)$ is a parameter that will be chosen below. Thus, $\widetilde w_a\in H^2\cap  V$ is a function satisfying the inequality
\begin{equation} \label{wa-est}
	\|\widetilde w_a-w_a\|\le C_{12}\gamma\,\|w_0\|_1,\quad 
	\|\widetilde w_a\|_2\le C_{13}(\gamma)\|w_0\|_1,
\end{equation}
where we used~\eqref{2.63}, \eqref{omegav}, and the boundedness of~$\Omega_\gamma$ from~$V$ to~$H^2$. Let us consider problem~\eqref{2.46}, \eqref{2.47}. We shall need the two results below. The first one deals with the regularity and an a priori estimate for solutions of~\eqref{2.46}. Given a time interval $i'\subset\R$, let us define the space 
\begin{equation} \label{2.66}
\ZZ(I'):=\{v\in L^2(I',V\cap H_\sigma^3): \p_t v\in L^2(I',V), \p_t^2 v\in L^2(I',V^*)\},
\end{equation}
where the functional spaces in~$x$ are considered on the domain~$D'$. The proof of the following result is rather standard and will be given in Section~\ref{proof-p26}. 

\begin{proposition} \label{l2.6}
Let $I=[a,1]$, let $Q=I\times D'$, and let $\RR$ be a mapping that takes a triple $(\tilde u,\widetilde w_a,f)$ to the solution~$\widetilde w$ of problem~\eqref{2.46}, \eqref{2.47}. Then~$\RR$ acts from $B_\UU(\rho)\times (H^2\cap V)\times H^1(Q)$ to the space~$\ZZ(I)$ and is a $C^1$ function of its argument that is bounded on bounded subsets, together with its derivatives of the first order. Moreover, for any $K>0$ there is $C_K>0$ such that, for $\tilde u\in B_\UU(\rho)$, $\widetilde w_a\in B_{V\cap H^2}(K)$, $f\in B_{H^1(Q)}(K)$, and any interval $I_\theta=[\theta,1]$ with $a\le\theta<1$, we have 
	\begin{equation} \label{2.059}
		\|\RR(\tilde u,\widetilde w_a,f)\|_{\ZZ(I_\theta)}
		\le C_K\bigl(\|\widetilde w(\theta)\|_2+\|f\|_{H^1(I_\theta\times D')}\bigr).
	\end{equation}
\end{proposition}

The second result concerns a control problem for~\eqref{2.46}, \eqref{2.47} and is a consequence of Theorem~2 in~\cite{FGIP-2004} and Theorem~3.1 in~\cite{shirikyan-asens2015}  (see Remark~\ref{r2.6} below). 

\begin{proposition} \label{p2.5}
For any $\rho>0$ and $\e\in(0,1)$, there are positive numbers~$d$ and~$C$, and a continuous mapping\,\footnote{The mapping~$\CC_\e$ depends also on~$R$. However, we omit that dependence from the notation, because the parameter~$R$ will be fixed when applying Proposition~\ref{p2.5}.}
$$
\CC_\e:B_\UU(\rho)\to \LL(H,H_0^1(Q_0,\R^2)), \quad Q_0:=[a,b]\times D_0,  
$$ 
such that the following properties hold.
\begin{description}
\item[\underline{\rm Contraction:}]
For any $\tilde u\in B_\UU(\rho)$ and $\widetilde w_a\in B_H(d)$, the solution~$\widetilde w\in\ZZ(I)$ of problem~\eqref{2.46}, \eqref{2.47} with $f=\CC_\e(\tilde u)\widetilde w_a$ satisfies the inequality\,\footnote{The result established in~\cite{shirikyan-asens2015} claims only an estimate for the $L^2$-norm of the solution at time $t=b$: $\|\widetilde w(b)\|\le \e\|\widetilde w_a\|$. However, the regularising property of the Navier--Stokes flow implies that the $L^2$-norm on the left-hand side can be replaced with the $H^2$-norm for  $b\le t\le 1$; cf.\ proof of Proposition~\ref{l2.6}.}
\begin{equation} \label{3.02}
\|\widetilde w(t)\|_2\le \e\,\|\widetilde w_a\|
\quad\mbox{for $b\le t\le 1$}. 
\end{equation}
\item[\underline{\rm Regularity:}]
The mapping~$\CC_\e$ is infinitely smooth in the Fr\'echet sense. 
\item[\underline{\rm Lipschitz continuity:}]
The mapping $\CC_\e$ satisfies the inequality 
\begin{equation} \label{3.03}
\bigl\|\CC_\e(\tilde u_1)-\CC_\e(\tilde u_2)\bigr\|_{\LL}
\le C\,\|\tilde u_1-\tilde u_2\|_\UU,
\end{equation}
where $\|\cdot\|_\LL$ stands for the norm in the space $\LL(H,H_0^1(Q_0,\R^2))$. 
\end{description}
\end{proposition}

Let us fix a number~$\rho>0$ so large that $\|\tilde u\|_\UU\le \rho$ for any $u_0\in X$ and $\eta\in B_{G(\Sigma_0)}(R)$. Given~$\e>0$, we denote by~$d_\e>0$ the constant constructed in Proposition~\ref{p2.5} and choose $\delta>0$ so small that $(C_{12}+1)\delta\le d_\e$, so that (see~\eqref{2.63} and~\eqref{wa-est})
\begin{equation} \label{est-wa}
\|\widetilde w_a\|\le d_\e, \quad 
\|\widetilde w_a\|_2\le K:=C_{13}(\gamma)\delta. 
\end{equation}
Applying Propositions~\ref{p2.5} and~\ref{l2.6}, we construct a solution ${\widetilde w}^\e \in \ZZ(I)$ of problem~\eqref{2.46}, \eqref{2.47} with $\widetilde w_a\in B_H(d)\cap B_{H^2\cap V}(K)$ and $f=\CC_\e(\tilde u)\widetilde w_a$ such that inequality~\eqref{3.02} holds for $\widetilde w=\widetilde w^\e$.  

Let us denote by $\tilde \xi^\e$ the restriction of~$\widetilde w^\e$ to $I\times D'$, choose an arbitrary function $\chi\in C^\infty(\R)$ such that 
$$
0\le \chi\le 1,\quad 
\chi(t)=\left\{
\begin{array}{cl}
	1&\mbox{for $t\le b$},\\
	0&\mbox{for $t\ge c$},
\end{array}
\right.
$$
and extend (see~\eqref{2.61}) the function~$\xi$ to~$[a,1]$ by the relation $\xi(t)=\chi(t)\tilde\xi^\e(t)$. We claim that the mapping~$\varPhi'$ taking $(u_0,u_0',\eta)$ to~$\xi$ satisfies all required properties for an appropriate choice of the parameters~$\e$ and~$\gamma$. Indeed, let us denote by~$F$ the class of boundary functions $\zeta\in G$ such that $\zeta(t)=0$ for $0\le t\le a$ and~$\zeta|_{I}$ belongs to the space of restrictions to $I\times \p D'$ of the elements of~$\ZZ(I)$. Note that~$F$ has a natural structure of the quotient  (Banach) space and is compactly embedded into~$G_s$ for any $s\in (\frac32,2)$. The construction implies that $\xi\in F$. Furthermore, we have $\xi(t,x)=0$ for $t\notin[a,c]$ or $x\notin\Gamma$, and since $[a,c]\times \Gamma\subset \Sigma_0$, we conclude that $\xi\in G(\Sigma_0)$. To prove the regularity of the mapping~$\varPhi'$ with respect to~$\eta$, we note that its restriction to~$[a,1]$ can be written as 
\begin{gather} 
	\varPhi'(u_0,u_0',\eta)=\bigl(\chi(t)\RR(\tilde u,\widetilde w_a,\CC_\e(\tilde u)\widetilde w_a)
	\bigr)\big|_{\p D}\,,\label{phi}\\
	\tilde u=\LL\bigl(\SSS(u_0,\eta)\bigr),\quad 
	\widetilde w_a=\Omega_\gamma\bigl(S_a(u_0,\eta)-S_a(u_0',\eta)\bigr),\notag
\end{gather}
where~$\LL$ is the extension operator in Proposition~\ref{p4.9}, $\{\Omega_\gamma\}$ is the family of regularising operators in Proposition~\ref{p3.12}, and the function $S_a(u_0,\eta)-S_a(u_0',\eta)$ is extended to~$D'$ by zero. Since all the mappings that enter~\eqref{phi} are  $C^1$-smooth, so is~$ \varPhi'$. Thus, it remains to establish inequalities~\eqref{2.48} and~\eqref{2.49}. 

\medskip
{\it Step~5: Proof of~\eqref{2.48} and~\eqref{2.49}\/}. To estimate the norm of~$\xi=\varPhi'(u_0,u_0',\eta)$ in~$F$, we note that 
\begin{equation} \label{xiz}
	\|\xi\|_F\le C_{14}\|\widetilde w^\e\|_{\ZZ(I)}. 
\end{equation}
In view of~\eqref{2.059}, \eqref{3.03}, and the second inequality in~\eqref{est-wa}, we have 
\begin{equation} \label{we-est}
\|\widetilde w^\e\|_{\ZZ(I)}\le C_{15}(\gamma,\e)\|w_0\|_1. 
\end{equation}
Combining this with~\eqref{xiz}, we see that $\|\varPhi'(u_0,u_0',\eta)\|_F$ can be estimated by the right-hand side of~\eqref{2.48}. Differentiating~\eqref{phi} with respect to~$\eta$ and using the boundedness of the derivatives of~$\RR$ , $\SSS$, and~$\CC_\e$ on bounded subset, we can apply similar arguments to prove that $\|D_\eta \varPhi'(u_0,u_0',\eta)\|_{\LL(G)}$ also does not exceed right-hand side of~\eqref{2.48}. 

To establish~\eqref{2.49}, let us denote
$$
u(t)=S_t(u_0,\eta), \quad u'(t)=S_t(u_0,\eta+\varPhi'(u_0,u_0',\eta)), \quad u^\e(t)=u(t)+\widetilde w^\e(t)\big|_{D}, 
$$
where $\widetilde w^\e=\RR(\tilde u,\widetilde w_a,\CC_\e(\tilde u)\widetilde w_a)$ and $a\le t\le 1$ in the last relation. Then, in view of inequality~\eqref{3.02} and the Lipschitz-continuity of the resolving operator for the Navier--Stokes system considered on~$[a,\tau]$ (see part~(a) of Proposition~\ref{p3.1}), we can write 
\begin{align}
	\|u(\tau)-u'(\tau)\|
	&\le \|\widetilde w^\e(\tau)\|+\|u^\e(\tau)-u'(\tau)\|\notag\\
	&\le \e \|\widetilde w_a\|
	+C_{16}\bigl(\|w_a-\widetilde w_a\|+\|\xi-\tilde\xi^\e\|_{G_s}\bigr),
	\label{utau}
\end{align}
where $C_{16}>0$ does not depend on~$\e$ and~$\gamma$. Since $\chi(t)=1$ for $t\le b$ and $\CC_\e(\tilde u)\widetilde w_a$ is supported by~$[a,b]\times D_0$, using~\eqref{2.059} on the interval $I_b=[b,1]$ and inequality~\eqref{3.02}, we see that 
$$
\|\xi-\tilde\xi^\e\|_{G_s}\le C_{17}\|\widetilde w^\e\|_{\ZZ(I_b)}
\le C_{18}\|\widetilde w^\e(b)\|_2\le C_{18}\e\|\widetilde w_a\|.
$$
Combining this with~\eqref{utau}, \eqref{2.63}, and the first inequality in~\eqref{wa-est}, we derive
\begin{align*}
\|u(\tau)-u'(\tau)\|
&\le \e(1+C_{16}C_{18})\|\widetilde w_a\|+C_{16}\|w_a-\widetilde w _a\|\\
&\le2C_{12}(1+C_{16}C_{18})\e\|w_0\|_1+C_{16}C_{12}\gamma\|w_0\|_1. 
\end{align*}
Choosing $\gamma=(2C_{16}C_{12})^{-1}\varkappa$ and~$\e=(4C_{12}(1+C_{16}C_{18}))^{-1}\varkappa$, and taking~$\delta>0$ so small that~\eqref{est-wa} holds, we arrive at~\eqref{2.49}. This completes the proof of Theorem~\ref{t3.1}.

\begin{remark} \label{r2.6}
	Theorem~3.1 in~\cite{shirikyan-asens2015} was established under the hypothesis that the function~$\tilde u$ is a solution of the Navier--Stokes system. Namely, it was required that $\tilde u$ should belong to the space~$\ZZ(I)$ and, in particular, should vanish on the boundary~$\p D'$. However, the key ingredient of the proof---the observability inequality---remains valid if we only assume that~$\tilde u\in\UU$. This can be seen by analysing the proof of Lemma~1 in~\cite{FGIP-2004}, which is the main step in the proof of the local exact controllability (see Theorem~2 in~\cite{FGIP-2004}). 
\end{remark}

\subsection{Proof of Proposition~\ref{l2.6}}
\label{proof-p26}

We confine ourselves to the proof of inequality~\eqref{2.059} in the case $\theta=a$. The remaining assertions are standard facts of the general theory of nonlinear PDEs (cf.~\cite{kuksin-1982} and~\cite[Chapter~1]{VF1988}). 

Projecting Eq.~\eqref{2.46} to the space~$H$ over~$D'$, we reduce it to the evolution equation
\begin{equation} \label{NS-pr}
	\dot w+\nu Lw+B(w)+B(u,w)+B(w,u)=\Pi f,
\end{equation}
where $L=-\Pi\Delta$, $B(u,w)=\Pi(\langle u,\nabla\rangle w)$, $B(w)=B(w,w)$, and we write $w$ and~$u$ instead~$\widetilde w$ and~$\tilde u$ to simplify notation. The proof of~\eqref{2.059} is divided into several (standard) steps; cf.\ proof of Theorem~6.2 in~\cite{BV1992}. 

\medskip
{\it Step~1. Estimate in $L^2(V)\cap C(H)$}. 
Taking the inner product in~$L^2$ of~\eqref{NS-pr} with~$2w$ and using the relation $(B(v,w),w)=0$, we derive
\begin{equation} \label{L2-1}
	\p_t\|w\|^2+2\nu\|w\|_1^2=2(f,w)+2(B(w,u),w).
\end{equation}
It follows from H\"older's inequality and a well-known estimate for the quadratic term~$B$ that
\begin{equation} \label{L2-2}
	2|(f,w)+(B(w,u),w)|\le\nu\|w\|_1^2+C_1\bigl(\|w\|^2+\|f\|^2),
\end{equation}
where we denote by~$C_i$ positive numbers depending only on~$\nu$, $\rho$, and~$K$. Combining~\eqref{L2-1} and~\eqref{L2-2}, we derive 
$$
\p_t\|w\|^2+\nu\|w\|_1^2\le C_1\bigl(\|w\|^2+\|f\|^2).
$$
Application of Gronwall's inequality results in 
\begin{equation} \label{est-L2}
	\|w\|_{C(I_a,H)}+\|w\|_{L^2(I_a,V)}\le C_2\bigl(\|w(a)\|+\|f\|_{L^2(Q)}\bigr). 
\end{equation}

\smallskip
{\it Step~2. Estimate in $L^2(H^2)\cap W^{1,2}(H)$}. 
Let us take the inner product in~$L^2$ of Eq.~\eqref{NS-pr} with~$2Lw$. Using the inequalities
\begin{align*}
	\bigl|(B(w),Lw)|&\le C_3\|w\|_\infty \|w\|_1\|w\|_2
	\le C_4\|w\|^{1/2}\|w\|_1\|w\|_2^{3/2}\\
	&\le \tfrac{\nu}{8}\|w\|_2^2+C_5\|w\|^2\|w\|_1^4,\\
	\bigl|(B(u,w),Lw)|&\le C_3\|u\|_\infty \|w\|_1\|w\|_2
	\le C_4\|u\|_2\|w\|_1\|w\|_2 \\
		&\le \tfrac{\nu}{8}\|w\|_2^2+C_5\|u\|_2^2\|w\|_1^2,\\
	\bigl|(B(w,u),Lw)|&\le C_3\|w\|_\infty \|u\|_1\|w\|_2
	\le C_4\|w\|^{1/2}\|u\|_1\|w\|_2^{3/2} \\
		&\le \tfrac{\nu}{8}\|w\|_2^2+C_5\|u\|_1^4\|w\|^2
\end{align*}
and carrying out some simple transformations, we derive
$$
\p_t\|w\|_1^2+\nu\|w\|_2^2\le C_6\bigl(\|w\|^2\|w\|_1^2+\|u\|_2^2\bigr)\|w\|_1^2+C_6\|u\|_1^4\|w\|^2+C_6\|f\|^2. 
$$
Applying Gronwall's inequality and using~\eqref{est-L2}, we derive
\begin{equation} \label{est-H1}
	\|w\|_{C(I_a,V)}+\|w\|_{L^2(I_a,H^2)}
	\le C_7\bigl(\|w(a)\|_1+\|f\|_{L^2(Q)}\bigr). 
\end{equation}
Furthermore, it follows from~\eqref{NS-pr} that 
\begin{align*}
\|\dot w\|^2&\le C_8\bigl(\|w\|_2^2+\|f\|^2+\|B(w)+B(u,w)+B(w,u)\|^2\bigr)\\
&\le C_8\bigl(\|w\|_2^2+\|f\|^2\bigr)+C_9\bigl(\|w\|_2^2\|w\|_1^2+\|u\|_2^2\|w\|_1^2+\|w\|_2^2\|u\|_1^2\bigr).
\end{align*}
Combining this with~\eqref{est-H1}, we see that 
\begin{equation} \label{est-dtL2}
	\|\dot w\|_{L^2(Q)}
	\le C_{10}\bigl(\|w(a)\|_1+\|f\|_{L^2(Q)}\bigr). 
\end{equation}

\smallskip
{\it Step~3. Estimate in $L^\infty(W^{1,q})$}. 
Let us rewrite Eq.~\eqref{NS-pr} in the form
$$
\dot w+\nu Lw=h(t,x),
$$
where we set (cf.~\eqref{2.20})
$$
h=h_1+h_2, \quad h_1=\Pi f, \quad 
h_2=-\bigl(B(w)+B(u,w)+B(w,u)\bigr).
$$
If we prove that, for any $q\in(2,\infty)$,
\begin{equation} \label{h12}
\|h\|_{L^2(I_a,L^q)}\le C_{11}\bigl(\|f\|_{H^1(Q)}+\|w\|_{C(I_a,V)}+\|w\|_{L^2(I_a,H^2)}\bigr),	
\end{equation}
then the argument in the proof of Proposition~\ref{p3.1}~(b) combined with~\eqref{est-H1} will show that
\begin{align} 
	\|w\|_{L^\infty(I_a,W^{1,q})}
	&\le C_{12}\bigl(\|w(a)\|_{W^{1,q}}+\|f\|_{H^1(Q)}\bigr)\notag\\
	&\le C_{13}\bigl(\|w(a)\|_2+\|f\|_{H^1(Q)}\bigr). \label{LW1q}
\end{align}
The continuity of Leray's projection in the $L^q$ norm implies that  $\|h_1\|_{L^q}\le \|f\|_1$, so that we establish only a bound for the norm of~$h_2$. It follows from H\"older's inequality and the continuous embedding $H^1\subset L^{2q}$ that
\begin{align*}
\|B(v_1,v_2)\|_{L^q}
&\le  \|v_1\|_{L^{2q}}\|\nabla\otimes v_2\|_{L^{2q}}
\le C_{14}\|v_1\|_1\|v_2\|_2, 
\end{align*}
whence we see that
$$
\|B(v_1,v_2)\|_{L^2(I_a,L^q)}\le C_{14}\|v_1\|_{C(I_a,H^1)}\,\|v_2\|_{L^2(I_a,H^2)}. 
$$
This implies the required bound~\eqref{h12} for~$h_2$. 

\smallskip
{\it Step~4. Estimate in $W^{2,2}(V^*)$}. 
Differentiating~\eqref{NS-pr} in time, we derive
\begin{equation} \label{NS-diff}
	\dot z+\nu Lz+Q(u,\dot u,w,z)=g(t),
\end{equation}
where $z=\p_tw$, $g=\p_t(\Pi f)$, and 
$$
Q(u,\dot u,w,z)=B(z,w)+B(w,z)+B(z,u)+B(u,z)+B(\dot u,w)+B(w,\dot u).
$$
Let us take the inner product in~$L^2$ of Eq.~\eqref{NS-diff} with the function~$2z$. Since $(B(v,z),z)=0$, we derive 
\begin{multline} \label{NS-zz}
	\p_t\|z\|^2+2\nu \|z\|_1^2=2(g,z)-2(B(z,w),z)\\
	+2(B(z),u)+2(B(\dot u,z),w)+2(B(w,z),\dot u).
\end{multline}
Now note that 
\begin{align*}
	|(g,z)|&\le \|g\|\,\|z\|\le \|f\|_1^2+\|z\|^2,\\
	|(B(z,w),z)|&\le C_{15}\|z\|\,\|z\|_1\|w\|_1
	\le\tfrac{\nu}{8}\|z\|_1^2+C_{16}\|w\|_1^2\|z\|^2,\\
	|(B(z),u)|&\le C_{15}\|z\|\,\|z\|_1\|u\|_{L^\infty}
	\le\tfrac{\nu}{8}\|z\|_1^2+C_{16}\|z\|^2,\\
	|(B(\dot u,z),w)|&\le C_{15}\|\dot u\|\,\|z\|_1\|w\|_{L^\infty}
	\le\tfrac{\nu}{8}\|z\|_1^2+C_{16}\|\dot u\|^2\|w\|_{W^{1,q}}^2,\\
	|(B(w,z),\dot u)|&\le C_{15}\|w\|_{L^\infty}\|z\|_1\|\dot u\|
	\le\tfrac{\nu}{8}\|z\|_1^2+C_{16}\|\dot u\|^2\|w\|_{W^{1,q}}^2,
\end{align*}
where the last two estimates use the continuous embedding $W^{1,q}(D')\subset L^\infty(D')$ valid for $q>2$. Substituting these inequalities into~\eqref{NS-zz} and recalling~\eqref{est-H1} to estimate~$\|w\|_1$, we derive 
$$
\p_t\|z\|^2+\nu \|z\|_1^2\le C_{17}\bigl(\|z\|^2+\|f\|_1^2+\|\dot u\|^2\|w\|_{W^{1,q}}^2\bigr).
$$
Relation~\eqref{NS-pr} implies that $\|z(a)\|\le C_{18}\|w(a)\|_2$. 
Applying Gronwall's inequality and using~\eqref{LW1q}, we obtain 
\begin{equation} \label{est-dtw1}
	\|\p_t w\|_{C(I_a,L^2)}+\|\p_t w\|_{L^2(I_a,H^1)}
	\le C_{19}\bigl(\|w(a)\|_2+\|f\|_{H^1(Q)}\bigr). 
\end{equation}
Finally, resolving~\eqref{NS-diff} with respect to~$\dot z$ and taking the norm in~$V^*$, we easily conclude that $\|\p_t w\|_{L^2(I_a,V^*)}$ can be estimated by the right-hand side of~\eqref{est-dtw1}; cf.\ the derivation of~\eqref{est-dtL2}.  Thus, to complete the proof of~\eqref{2.059}, it remains to estimate the norm of~$w$ in~$L^2(I_a,H^3)$. 

\smallskip
{\it Step~5. Estimate in $L^2(H^3)$}. Resolving~\eqref{NS-pr} with respect to~$Lw$ and using the elliptic regularity for the Stokes operator~$L$, we see that 
\begin{equation} \label{wl3}
\|w(t)\|_3\le C_{20}\bigl(\|f(t)\|_1+\|B(w)+B(u,w)+B(w,u)\|_1\bigr).	
\end{equation}
To estimate the second term on the right-hand side, we note that
\begin{align*}
\|B(v_1,v_2)\|_1
&\le C_{21}\bigl(\|v_1\|_{W^{1,4}}\|v_2\|_{W^{1,4}}
+\|v_1\|_{L^\infty}\|v_2\|_2\bigr)\\
&\le C_{22}(\|v_1\|\,\|v_2\|_2+\|v_2\|\,\|v_1\|_2) 
+  C_{21} \|v_1\|_{L^\infty}\|v_2\|_2, 
\end{align*}
whence it follows that 
\begin{align*}
\|B(v_1,v_2)\|_{L^2(I_a,H^1)}
&\le C_{23}\sum_{i=1}^2\|v_i\|_{C(I_a,H^1)}\|v_{2-i}\|_{L^2(I_a,H^2)}\\
&+C_{23}\|v_1\|_{L^\infty(Q)}\|v_2\|_{L^2(I_a,H^2)}. 
\end{align*}
Substitution of this inequality into the right-hand side of~\eqref{wl3} results in
$$
\|w\|_{L^2(I_a,H^3)}\le C_{24}\bigl(\|w\|_{C(I_a,H^1)}+\|w\|_{L^2(I_a,H^2)}+\|w\|_{L^\infty(Q)}\bigr).
$$
Recalling~\eqref{est-H1} and~\eqref{LW1q}, we see that~$\|w\|_{L^2(I_a,H^3)}$ can be estimated by the right-hand side of~\eqref{est-dtw1}. This completes the proof of~\eqref{2.059}. 

\section{Appendix}
\label{s4}
\subsection{Sufficient conditions for mixing}
\label{s4.1}
Consider a discrete-time Markov process $(u_k,\IP_u)$ in a compact metric space~$X$. Let~$P_k(u,\Gamma)$ be the transition function for $(u_k,\IP_u)$ and let~$\PPPP_k$ and~$\PPPP_k^*$ be the corresponding Markov semigroups.  In this section, we recall a result on the uniqueness of a stationary measure for~$(u_k,\IP_u)$ and its exponential stability  in the dual-Lipschitz metric. 

\smallskip
Let us define the product space $\XXX=X\times X$ and denote by~$\Pi,\Pi':\XXX\to X$ the natural projections to its components, taking a point $\uuu=(u,u')\in\XXX$ to~$u$ and~$u'$, respectively. A Markov process $(\uuu_k,\bP_{\!\uuu})$ with the phase space~$\XXX$ is called an {\it extension\/} for $(u_k,\IP_u)$ if, for all $k\ge0$ and  $ \uuu=(u,u')\in\XXX$, we have
\begin{equation} \label{4.4}
\Pi_*\PPP_k(\uuu,\cdot)=P_k(u,\cdot), \quad \Pi_*'\PPP_k(\uuu,\cdot)=P_k(u',\cdot),
\end{equation}
where $\PPP_k(\uuu,\GGamma)$ stands for the transition function of $(\uuu_k,\bP_{\!\uuu})$. We have the following theorem established in~\cite{shirikyan-bf2008} (see also Section~3.1.3 in~\cite{KS-book}). 

\begin{theorem} \label{t4.2}
Let $X$ be a compact metric space and let $(u_k,\IP_u)$ be a family of discrete-time Markov processes in~$X$ that possesses an extension~$(\uuu_k,\bP_\uuu)$ satisfying the following properties for some closed subset $\BBB\subset\XXX$. 

\smallskip
\noindent
{\bf Recurrence:} 
The Markov time $\tau(\BBB)=\min\{k\ge0:\uuu_k\in\BBB\}$ 
is $\bP_\uuu$-almost surely finite for any $\uuu\in\XXX$, and there are positive numbers~$\beta$ and~$C_1$ such that
\begin{equation} \label{4.10}
\bE_\uuu e^{\beta\tau(\BBB)}\le C_1\quad\mbox{for any $\uuu\in\XXX$}. 
\end{equation}

\noindent
{\bf Squeezing:} 
There are positive numbers $q<1$, $d$, $\delta_1$, $\delta_2$, and~$C_2$ such that the Markov time $\sigma=\min\{k\ge0:d(u_k,u_k')>q^kd\}$ satisfies the inequalities
\begin{equation} \label{4.11}
\bP_\uuu\{\sigma=+\infty\}\ge\delta_1,\quad 
\bE_\uuu \bigl(e^{\delta_2\sigma}I_{\{\sigma<\infty\}}\bigr)\le C_2\quad\mbox{for $\uuu\in\BBB$}.
\end{equation}
Then $(u_k,\IP_u)$ has a unique stationary measure $\mu\in\PP(X)$, which is exponentially mixing for the dual-Lipschitz metric in the sense that~\eqref{4.12} holds for some positive constants~$\gamma$ and~$C$.
\end{theorem}

\subsection{Image of measures under regular mappings}
\label{s4.2}
 Let~$E$ be a separable Banach space represented as the direct sum of two closed subspaces~$F$ and~$F'$,
\begin{equation} \label{4.15}
E=F\dotplus F',
\end{equation}
where $\dim F<\infty$. We denote by~$\mathsf P$ and~$\mathsf P'$ the projections associated with~\eqref{4.15}. Let $\ell\in\PP(E)$ be a measure that has a bounded support and can be written as the tensor product of its marginals $\ell_F={\mathsf P}_*\ell$ and $\ell_{F'}=({\mathsf P}')_*\ell$. We assume that~$\ell_F$ has a $C^1$-smooth density with respect to the Lebesgue measure on~$F$. A proof of the following result can be found in~\cite{shirikyan-asens2015} (see Proposition~5.6). 

\begin{proposition} \label{p4.4}
In addition to the above hypotheses, assume that $\varPsi:E\to E$ is a mapping written in the form $\varPsi(\zeta)=\zeta+\varPhi(\zeta)$, where $\varPhi:E\to E$ is a $C^1$-smooth mapping such that $\varPhi(E)\subset F$ and 
\begin{equation} \label{4.16}
\|\varPhi(\zeta_1)\|\le\varkappa, \quad 
\|\varPhi(\zeta_1)-\varPhi(\zeta_2)\|\le\varkappa\,\|\zeta_1-\zeta_2\|
\quad\mbox{for all $\zeta_1,\zeta_2\in E$}, 
\end{equation}
where $\varkappa>0$ is a number. Then
\begin{equation} \label{4.17}
\|\ell-\varPsi_*(\ell)\|_{\mathrm{var}}\le C\varkappa,
\end{equation}
where $C>0$ does not depend on~$\varkappa$. 
\end{proposition}

\subsection{Measurable coupling associated with a cost}
\label{s4.3}
Let~$X$ be a compact subset of a separable Banach space~$H$. For any $\e>0$, we consider a function $d_\e:X\times X\to \R$ given by
$$
d_\e(u_1,u_2)=\left\{
\begin{array}{cl}
1 & \mbox{if $\|u_1-u_2\|>\e$},\\ 
0 & \mbox{if $\|u_1-u_2\|\le\e$},
\end{array}
\right.
$$
where $\|\cdot\|$ is the norm in~$H$. Given two measures $\mu_1,\mu_2\in\PP(X)$, we define the {\it cost associated with~$d_\e$\/} by the relation 
\begin{align}
C_\e(\mu_1,\mu_2)&=\inf_{\MMMM\in\Pi(\mu_1,\mu_2)} 
\int\limits_{X\times X}d_\e(u_1,u_2)\,\MMMM(\dd u_1,\dd u_2)\notag\\
&=\inf_{\MMMM\in\Pi(\mu_1,\mu_2)} 
\MMMM(\DDD_\e^c), \label{4.31}
\end{align}
where $\Pi(\mu_1,\mu_2)$ stands for the set of measures on~$X\times X$ with marginals~$\mu_1$ and~$\mu_2$, and $\DDD_\e=\{(u_1,u_2)\in X\times X: \|u_1-u_2\|\le\e\}$. Kantorovich's celebrated theorem claims that the infimum in~\eqref{4.31} is always achieved; see Theorem~5.10 in~\cite{villani2009}.

Now let $(Z,\ZZ)$ be a measurable space and let  $\{\mu^z,z\in Z\}\subset\PP(X)$ be a family of measures. Recall that~$\{\mu^z\}$ is called a {\it random probability measure\/} on~$X$ if the function $z\mapsto \mu^z(\Gamma)$ from~$Z$ to~$\R$ is measurable  for any $\Gamma\in\BB(X)$. 
The following result is a simple consequence of Corollary~5.22 in~\cite{villani2009}; its proof can be found in~\cite{shirikyan-asens2015} (see Proposition~5.3). 

\begin{proposition} \label{p4.5}
Let $\{\mu_i^z,z\in Z\}$, $i=1,2$ be two random probability measures on~$X$, let $\e:Z\to\R$ be a positive measurable function, and let $\theta\in(0,1)$. Then there is a probability space $(\Omega,\FF,\IP)$ and measurable functions $\xi_i^z(\omega)$, $i=1,2$, from $\Omega\times Z$ to~$X$ such that, for any $z\in Z$, the law~$\MMMM^z$ of $(\xi_1^z,\xi_2^z)$ belongs to~$\Pi(\mu_1^z,\mu_2^z)$ and satisfies the inequality
\begin{equation} \label{4.32}
\int\limits_{X\times X}
d_{\e(z)}(u_1,u_2)\MMMM^z(\dd u_1,\dd u_2)\le C_{\theta\e(z)}(\mu_1^z,\mu_2^z).
\end{equation}
\end{proposition}

We now formulate a simple result providing an estimate for $C_\e(\mu_1,\mu_2)$. Its proof is based on the Kantorovich duality (see Theorem~5.10 in~\cite{villani2009}) and can be found in~\cite{shirikyan-asens2015} (see Proposition~5.2). 

\begin{lemma} \label{l4.6}
Let~$\mu_1$ and~$\mu_2$ be two probability measures on a compact metric space~$X$ that are the laws of some random variables~$U_1$ and~$U_2$ defined on a probability space $(\Omega,\FF,\IP)$.  Suppose there is a measurable mapping $\varPsi:\Omega\to\Omega$ satisfying
\begin{equation} \label{4.33}
d_X\bigl(U_1(\omega),U_2(\varPsi(\omega))\bigr)\le \e
\quad\mbox{for a.e.~$\omega\in\Omega$},
\end{equation}
where $d_X$ is the metric on~$X$ and~$\e>0$ is a number. Then
\begin{equation} \label{4.34}
C_\e(\mu_1,\mu_2)\le 2\,\|\IP-\varPsi_*(\IP)\|_{\mathrm{var}}. 
\end{equation}
\end{lemma}

\subsection{Restriction to and extension from the boundary}
\label{s4.4}
Let $D\subset\R^2$ be a bounded domain that has an infinitely smooth boundary~$\p D$ and satisfies the hypotheses mentioned in the beginning of Section~\ref{s3.0}. Given a real number $s\ge1$, we write 
\begin{equation} \label{4.61}
\XX_s=\bigl\{u\in L^2(J,H_\sigma^{s+1}): \p_t u\in L^2(J,H_\sigma^{s-1})\bigr\},
\end{equation}
where $J=[0,1]$, $H^s(D,\R^2)$ is the usual Sobolev space of order~$s$ and~$H_\sigma^s$ denotes the space of divergence-free vector fields in~$H^s(D,\R^2)$. Recall that, for $s\ge3/2$, we also defined the space~$G_s$ of functions $v\in L^2(J,H^{s+1/2}(\p D))$ such that $\p_t v\in L^2(J,H^{s-3/2}(\p D))$ and\,\footnote{Let us note that if the equality in~\eqref{4.35} holds for a.e.~$t\in J$, then the continuity of~$v$ from~$J$ to~$ L^2(\p D,\R^2)$ implies it is true for all~$t\in J$.}
\begin{equation} \label{4.35}
\int_{\p D}\langle v(t),\nnn_x\rangle\dd\sigma=0\quad\mbox{for $t\in J$}. 
\end{equation}
The space~$\XX_s$ and~$G_s$ are endowed with the natural Hilbert structures and the corresponding norms. The following proposition gives a characterisation of traces of the functions in~$\XX_s$ to the lateral boundary $\Sigma=J\times\p D$. Its proof can be found in the paper~\cite{FGH-2002} (see Theorems~2.1 and~2.2), where the more complicated 3D case is discussed. For the reader's convenience, we reproduce here a complete proof in the 2D case, establishing an additional property. 

\begin{proposition} \label{p4.7}
For any integer $s\ge2$, the operator~$R$ taking $u\in\XX_s$ to its restriction to~$\Sigma$ is continuous from~$\XX_s$ to~$G_s$ and possesses a right inverse in the following sense: there is a continuous operator $Q:G_2\to\XX_2$ such that $RQv=v$ for $v\in G_2$,  and for any integer $s\ge2$, we have
\begin{equation} \label{4.36}
\|Qv\|_{\XX_s}\le C_s\|v\|_{G_s}\quad
\mbox{for $v\in G_s$}, 
\end{equation}
where $C_s>0$ does not depend on~$v$. 
\end{proposition}

\begin{proof}
The standard trace theorem for Sobolev spaces (e.g., see Chapter~4 in~\cite{adams1975})  implies that, for any $u\in\XX_s$, we have $u|_\Sigma\in L^2(J,H^{s+1/2}(\p D))$ and $\p_t(u|_\Sigma)\in L^2(J,H^{s-3/2}(\p D))$, and the corresponding norms are bounded by~$\|u\|_{\XX_s}$. Furthermore, since $\diver u=0$ in $J\times D$, and the function $t\mapsto v(t)$ is continuous from~$J$ to~$H^1(D,\R^2)$, we have $\int_D\diver u(t)\,\dd x=0$ for $t\in J$, whence it follows that~\eqref{4.35} holds. Thus, the restriction operator $R:\XX_s\to G_s$ is continuous. To construct its right inverse, we shall need the  lemma below. For $r\ge0$, let us denote by~$\dot H^r=\dot H^r(\p D,\R^2)$ the space of vector functions $v:\p D\to\R^2$ that belong to the Sobolev space of order~$r$ and satisfy the relation $\int_{\p D}\langle v,\nnn_x\rangle\dd\sigma=0$. 

\begin{lemma} \label{l4.8}
There is a continuous operator $\SE:\dot H^{1/2}(\p D,\R^2)\to H_\sigma^1(D,\R^2)$ such that the restriction of~$\SE v$ to~$\p D$ coincides with~$v$. Moreover, for any integer $s\ge1$ there is $C_s'>0$ such that 
\begin{equation} \label{4.37}
\|\SE v\|_s\le C_s'\|v\|_{s-1/2}\quad\mbox{for any $v\in \dot H^{s-1/2}$}. 
\end{equation}
\end{lemma}

Taking this lemma for granted, let us complete the proof of the proposition. Let us fix $v\in G_2$. It follows from~\eqref{4.35} that $v(t)\in \dot H^{5/2}$ for $t\in J$. We can thus define a function~$u(t,x)$ by the relation $u(t,\cdot)=\SE v(t,\cdot)$ for $t\in J$, where~$\SE $ is the operator in Lemma~\ref{l4.8}. By continuity of~$\SE $, we have $u\in L^2(J,H_\sigma^3)$. Moreover, since $\p_t (\SE v(t))=\SE (\p_tv(t))$, we see that $\p_tu\in L^2(J,H^1)$, so that $u\in \XX_2$. The above argument also shows that~\eqref{4.36}  holds for $s=2$. Finally, it follows from~\eqref{4.37} that~\eqref{4.36} is valid for any $s\ge2$. This completes the proof of the proposition. 
\end{proof}

\begin{remark}
The proof of Proposition~\ref{p4.7} implies that if $v=0$ in a region $[\alpha,\beta]\times\p D$, then~$Qv$ vanishes in $[\alpha,\beta]\times D$.
\end{remark}

\begin{proof}[Proof of Lemma~\ref{l4.8}]
To make the main idea more transparent, we first consider the case in which~$D$ is simply-connected. Let us fix a function $v\in \dot H^{1/2}(\p D,\R^2)$ and write it in the form
\begin{equation} \label{4.38}
v=v_n\nnn_x+v_\tau\ttau_x, \quad v_n(x)=\langle v(x),\nnn_x\rangle, \quad
v_\tau(x)=\langle v(x),\ttau_x\rangle,
\end{equation}
where $\nnn_x$ and~$\ttau_x$ are the unit (outward) normal and tangent vectors at a point $x\in\p D$ chosen so that $(\nnn_x,\ttau_x)$ is a positively oriented basis of~$\R^2$. We shall construct two vector functions~$u_n$ and~$u_\tau$ belonging to~$H_\sigma^1(D,\R^2)$ such that
\begin{equation} \label{4.39}
\langle u_n,\nnn\rangle\bigr|_{\p D}=v_n, \quad
\langle u_\tau,\nnn\rangle\bigr|_{\p D}=0, \quad
\langle u_\tau,\ttau\rangle\bigr|_{\p D}=v_\tau-\langle u_n,\ttau\rangle\bigr|_{\p D}.
\end{equation}
The operator~$\SE $ is then defined by $\SE v=u_n+u_\tau$. Moreover, the construction will imply that~\eqref{4.37} is also satisfied. 

\smallskip
{\it Step 1: Construction of~$u_n$\/}. We seek~$u_n$ in the form $u_n=\nabla^\bot p$, where $\nabla^\bot=(-\p_2,\p_1)$ and $p\in H^2(D)$. 
The first relation in~\eqref{4.39} can be rewritten in terms of the derivatives of~$p$ and the tangent vector~$\ttau$ as follows:
$$
\frac{\p p}{\p\ttau}\Bigr|_{\p D}=-v_n. 
$$
Since $\int_{\p D}v_n\dd\sigma=0$, we can find a function $w\in H^{3/2}(\p D)$ such that 
\begin{equation} \label{4.40}
\frac{\p w}{\p\ttau}=-v_n,\quad \|w\|_{H^{3/2}}\le C_1\|v_n\|_{H^{1/2}}. 
\end{equation}
Let $p\in H^2(D)$ be a harmonic function in~$D$ such that $p|_{\p D}=w$. Then
\begin{equation} \label{4.41}
\|p\|_{H^2}\le C_2\|w\|_{H^{3/2}}.
\end{equation}
Combining~\eqref{4.40} and~\eqref{4.41}, we see that the function $u_n=(-\p_2p,\p_1p)$ satisfies the required properties. Moreover, the construction implies that 
\begin{equation} \label{4.42}
\|u_n\|_{H^s}\le C_3\|v_n\|_{H^{s-1/2}}\quad\mbox{for any integer $s\ge1$},
\end{equation}
where $C_3>0$ depends only on~$s$.

\smallskip
{\it Step 2: Construction of~$u_\tau$\/}. The required function is sought in the form $u_\tau=\nabla^\bot q$, where $q\in H^2(D)$ is an unknown function. 
Let us note that the function~$\tilde v_\tau$ defined by the right-hand side of the third relation in~\eqref{4.39} belongs to the space~$H^{1/2}(\p D)$ and satisfies the inequality 
\begin{equation} \label{4.45}
\|\tilde v_\tau\|_{H^{s-1/2}}\le C_4\|v\|_{H^{s-1/2}}\quad\mbox{for any integer $s\ge1$},
\end{equation}
where $C_4>0$ depends only on~$s$. Furthermore, the second and third relations in~\eqref{4.39} with $u_\tau=(-\p_2 q,\p_1q)$ are equivalent to
\begin{equation} \label{4.43}
q\bigr|_{\p D}=C, \quad \frac{\p q}{\p\nnn}\Bigr|_{\p D}=\tilde v_\tau,
\end{equation}
where $C\in\R$ is a number. Since~$u_\tau$ is obtained by differentiating~$q$, we can take $C=0$. The elliptic equation $\Delta^2q=0$ supplemented with the boundary conditions~\eqref{4.43} has a unique solution $q\in H^2(D)$, and the elliptic regularity implies that $\|q\|_{H^{s+1}}\le C_5\|\tilde v_\tau\|_{H^{s-1/2}}$. Combining this with~\eqref{4.45}, we see that the function $u_\tau=(-\p_2q,\p_1 q)$ possesses all required properties. 

\smallskip
{\it Step 3: General case\/}.
The construction of~$u_\tau$ in Step~2 does not use the assumption that~$D$ should be simply-connected. We thus need only to extend the argument of Step~1 to the case of an arbitrary domain satisfying the hypotheses of Section~\ref{s3.0}. 

Let us denote by~$\Gamma_i$ the boundary of the domain~$D_i$ and by~$\widetilde\Gamma$ that of~$\widetilde D$. For $1\le i\le m$, let $z_i\in C^\infty(\overline D)$ be a harmonic function in~$D$ with zero mean value such that 
\begin{equation} \label{4.44}
\frac{\p z_i}{\p\nnn}\Bigr|_{\Gamma_i}=1, \qquad
\frac{\p z_i}{\p\nnn}\Bigr|_{\Gamma_j}=0\quad\mbox{for $j\ne i$}, 
\qquad \frac{\p z_i}{\p\nnn}\Bigr|_{\widetilde\Gamma}
=-|\Gamma_i|/|\widetilde\Gamma|,
\end{equation}
where $|\gamma|$ stands for the length of a curve~$\gamma$. It is straightforward to check that the boundary conditions~\eqref{4.44} satisfy the compatibility condition for the existence of a solution of the Neumann problem for the Laplace equation (see Proposition~7.7 in~\cite[Chapter~5]{taylor1996}), so that the functions~$z_i$ are well defined. We seek~$u_n$ in the form
\begin{equation} \label{4.46}
u_n=\nabla^\bot p+\sum_{i=1}^mc_i\nabla z_i,
\end{equation}
where $p\in H^2(D)$ and $c_i\in\R$ are chosen below. The first relation in~\eqref{4.39} is equivalent to
\begin{align} 
-\frac{\p p}{\p\ttau}\Bigr|_{\Gamma_i}
&=v_n^{(i)}:=v_n\bigr|_{\Gamma_i}-c_i\quad\mbox{for $1\le i\le m$}, \label{4.47}\\
-\frac{\p p}{\p\ttau}\Bigr|_{\widetilde\Gamma}
&=\tilde v_n:=v_n\bigr|_{\widetilde\Gamma}+|\widetilde \Gamma|^{-1}
\sum_{i=1}^m c_i|\Gamma_i|.  \label{4.48}
\end{align}
Choosing $c_i=|\Gamma_i|^{-1}\int_{\Gamma_i}v_n\dd\sigma$, we see that 
\begin{equation} \label{4.49}
\int_{\widetilde\Gamma}\tilde v_n\dd\sigma=0, \quad
\int_{\Gamma_i}v_n^{(i)}\dd\sigma=0\quad\mbox{for $1\le i\le m$}, 
\end{equation}
where we used the fact that 
$$
0=\int_{\p D}v_n\dd\sigma=\int_{\widetilde\Gamma}v_n\dd\sigma
+\sum_{i=1}^m\int_{\Gamma_i}v_n\dd\sigma. 
$$
It follows from~\eqref{4.49} that there are functions $w_i\in H^{3/2}(\Gamma_i)$ and $\widetilde w\in H^{3/2}(\widetilde\Gamma)$ such that (cf.~\eqref{4.40})
\begin{gather} 
\frac{\p w_i}{\p\ttau}\Bigr|_{\Gamma_i}=-v_n^{(i)},\quad 
\frac{\p \widetilde w}{\p\ttau}\Bigr|_{\widetilde\Gamma}=-\tilde v_n,\label{4.50}\\
\sum_{i=1}^m\|w_i\|_{H^{s+1/2}}+\|\widetilde w\|_{H^{s+1/2}}\le C_6\|v_n\|_{H^{s-1/2}}. 
\label{4.51}
\end{gather}
Let $p\in H^2(D)$ be a harmonic function in~$D$ such that
$$
p\bigr|_{\Gamma_i}=w_i\quad\mbox{for $1\le i\le m$},
\qquad p\bigr|_{\widetilde\Gamma}=\widetilde w. 
$$
Combining this with~\eqref{4.47}, \eqref{4.48}, and~\eqref{4.50}, we see that the function~$u_n$ defined by~\eqref{4.46} with the above choice of~$c_i$  belongs to~$H_\sigma^1(D,\R^2)$ and satisfies the first relation in~\eqref{4.39}. Finally, it follows from~\eqref{4.51} that~\eqref{4.42} also holds. This completes the proof of the lemma.  
\end{proof}

\begin{remark} \label{r3.8}
Analysing the proof of Proposition~\ref{p4.7}, it is straightforward to see that the result remains true for any real number $s\in(\frac32,2)$. More precisely, the application $Q:G_2\to\XX_2$ can be extended by continuity to an operator $Q_s:G_s\to\XX_s$ for $s\in(\frac32,2)$ such that $RQ_sv=v$ for any $v\in G_s$. 
\end{remark}

We now consider a particular case of the above extension theorem when the mean value of the normal component of $v\in G_s$ is zero not only on the entire boundary, but also on each of the connected components. In this case, it is possible get an extension that satisfies an additional property. Namely, let us denote by~$G_s^0$ the space of functions $v\in G_s$ such that 
\begin{equation} \label{4.061}
\int_{\p D_i}\langle v(t),\nnn_x\rangle\dd\sigma=0\quad
\mbox{for $t\in J$ and $1\le i\le m$}.
\end{equation}
The following result is due to E.~Hopf, and its proof is essentially contained in Section~II.1.4 in~\cite{temam1979}, so that we only outline the corresponding argument.

\begin{proposition} \label{p4.8}
For any $\e>0$, there is a linear operator $Q_\e:G_2^0\to\XX_2$ such that, for any $v\in G_2^0$, the restriction of~$Q_\e v$ to $J\times\p D$ coincides with~$v$, inequality~\eqref{4.36} holds for $Q=Q_\e$ and a number~$C_s$ depending on~$\e$ and~$s$, and 
\begin{equation} \label{4.62}
\bigl|\bigl(\langle u,\nabla\rangle (Q_\e v)(t),u\bigr)_{L^2}\bigr|
\le \e\,\|v(t)\|_{3/2}\|u\|_1^2
\quad \mbox{for any $u\in V$, $t\in J$}. 
\end{equation}
\end{proposition}

\begin{proof}[Sketch of the proof]
As in the case of Proposition~\ref{p4.7}, it suffices to construct an extension operator in~$x$; see Lemma~\ref{l4.8}. Namely, let $H^{\frac12,0}(\p D,\R^2)$ be the subspace of those functions $v\in \dot H^{1/2}(\p D,\R^2)$ that satisfy the relations
$$
\int_{\p D_i}\langle v,\nnn_x\rangle\dd\sigma=0\quad
\mbox{for $1\le i\le m$}.
$$
We claim that, for any $\e>0$, there is a continuous linear extension operator 
$$
\SE _\e:H^{\frac12,0}(\p D,\R^2)\to H_\sigma^1(D,\R^2)
$$
that satisfies~\eqref{4.37} with $\SE =\SE _\e$ and any integer $s\ge1$, as well as the inequality
\begin{equation} \label{4.63}
\bigl|\bigl(\langle u,\nabla\rangle \SE _\e v,u\bigr)_{L^2}\bigr|
\le \e\,\|v\|_{3/2}\|u\|_1^2 
\end{equation}
for any $u\in V$ and $v\in (H^{\frac12,0}\cap H^{3/2})(\p D,\R^2)$. Once this is proved, one can conclude using the same argument as in the  proof of Proposition~\ref{p4.7}. 

Analysing the proof of Lemma~\ref{l4.8}, we see that the extension operator~$\SE $ constructed there possesses the following property: there is a continuous operator $\SE' :H^{\frac12,0}(\p D,\R^2)\to H^2(D)$ such that, for $v\in H^{\frac12,0}(\p D,\R^2)$, we have 
\begin{equation} \label{4.64}
\SE v=\nabla^\bot(\SE' v), \quad 
\|\SE' v\|_{s+1}\le C_s\|v\|_{s-1/2}
\quad\mbox{for any integer $s\ge1$}. 
\end{equation}
We now choose a function $\theta_\delta\in C^\infty(\overline{D})$ such that
\begin{align*}
\theta_\delta(x)&=1& &\mbox{for $d(x,\p D)\le\tfrac12 e^{-2/\delta}$}, \\
\theta_\delta(x)&=0&&\mbox{for $d(x,\p D)\ge 2e^{-1/\delta}$},\\
\bigl|\nabla\theta_\delta(x)\bigr|
&\le\frac{\delta}{d(x,\p D)}&
&\mbox{for $d(x,\p D)\le 2e^{-1/\delta}$},
\end{align*}
where $d(x,\p D)$ denotes the distance from~$x$ to~$\p D$; see Lemma~1.9 in~\cite[Section~II.1]{temam1979}. The operator~$\SE _\e$ is sought in the form
$$
\bigl(\SE _\e v\bigr)(x)=\nabla^\bot\bigl(\theta_\delta(x)(\SE' v)(x)\bigr),
$$
where $\delta=\delta(\e)>0$ is a number. Inequality~\eqref{4.37} follows  from~\eqref{4.64}, and a simple calculation based on Hardy's inequality shows that~\eqref{4.63} is true for sufficiently small~$\delta$; see the proof of Lemma~1.8 in~\cite[Section~II.1]{temam1979}.  
\end{proof}

\begin{wrapfigure}{r}{0.39\textwidth}
\vspace{-20pt}
\begin{center}
\begin{tikzpicture}
\draw[domain=0:360,smooth,samples=100] plot (\x:{2+sin(3*\x)/4+cos(2*\x)/6+cos(3*\x+25)/7});
\draw[domain=0:360,smooth,samples=100] plot (\x:{2+sin(3*\x)/4+cos(2*\x)/6+cos(3*\x+25)/7+1/3});
\draw[fill=gray!30!white] (0,1) circle(0.4) node {$D_1$}; 
\draw[fill=gray!30!white] (1,-0.1) circle(0.35) node {$D_2$}; 
\draw[fill=gray!30!white] (-0.5,-0.3) circle(0.45) node {$D_3$}; 
\draw (0,-1.3) node {$D$}; 
\draw (2,-1.8) node {$\widetilde\DD$}; 
\end{tikzpicture}
\end{center}
\captionof{figure}{The domain~$\DD$\label{pic3}}
\end{wrapfigure}

\subsection{Extension to a larger domain}
\label{s4.5} 

As before, we denote by $D\subset\R^2$ a bounded domain satisfying the hypotheses of Section~\ref{s3.0}. In Section~\ref{s3.2}, we used the fact that the functions in~$\XX_s$ can be extended to a larger domain. Namely, let $\widetilde \DD\subset\R^2$ be a simply-connected domain containing the closure of~$\widetilde D$ and let 
\begin{equation} \label{4.52}
\DD=\widetilde \DD\setminus\biggl(\,\bigcup_{i=1}^m\overline{D}_i\biggr);
\end{equation}
cf.~\eqref{3.2}. The following proposition shows how to extend the divergence-free vector fields from~$D$ to~$\DD$. 

\begin{proposition} \label{p4.9}
There is a continuous linear operator
$$
\LL:L_\sigma^2(D,\R^2)\to L_\sigma^2(\DD,\R^2)
$$
possessing the following properties.
\begin{itemize}
\item[\bf(a)] 
For any $v\in L_\sigma^2(D,\R^2)$, the restriction of~$\LL v$ to~$D$ coincides with~$v$. 
\item[\bf(b)] 
If, in addition, $v\in W^{r,q}(D,\R^2)$ for some numbers $r\ge0$ and $q\in[2,\infty)$, then $\LL v$ belongs to~$W^{r,q}(\DD,\R^2)$ and satisfies the inequality
\begin{equation} \label{4.53}
\|\LL v\|_{W^{r,q}(\DD)}\le C_{r,q}\|v\|_{W^{r,q}(D)},
\end{equation}
where $C_{r,q}>0$ does not depend on~$v$.
\end{itemize}
\end{proposition}

Before proving this result, we state a straightforward corollary from it concerning the extension of functions belonging to~$\XX_s$. We denote by~$\XX_s(\DD)$ the space~$\XX_s$ constructed on the domain~$\DD$. 

\begin{corollary} \label{c4.10}
For any integer~$s\ge1$ and any $u\in\XX_s$, the function $\tilde u(t,x)$ defined by $\tilde u(t)=\LL u(t)$ belongs to~$\XX_s(\DD)$ and satisfies the inequality
\begin{equation} \label{4.54}
\|\tilde u\|_{\XX_s(\DD)}\le C_s \|u\|_{\XX_s}, 
\end{equation}
where $C_s>0$ does not depend on~$u$. If, in addition, $u\in C(J,W^{r,q}(D))$ for some $r\ge0$ and $q\in[2,\infty)$, then $\tilde u\in C(J,W^{r,q}(\DD))$, and we have 
\begin{equation} \label{4.054}
\|\tilde u\|_{C(J,W^{r,q}(\DD))}\le C_{r,q}' \|u\|_{C(J,W^{r,q}(D))}, 
\end{equation}
where $C_{r,q}'>0$ does not depend on~$u$. 
\end{corollary}

\begin{proof}[Proof of Proposition~\ref{p4.9}]
We first derive a necessary and sufficient condition for a function $w\in L_\sigma^2(D,\R^2)$ to be representable in the form 
\begin{equation} \label{4.55}
w=\nabla^\bot p=(-\p_2 p,\p_1 p),
\end{equation}
where $p\in \dot H^1(D)$, and $\dot H^k=\dot H^k(D)$ stands for the space of functions in~$H^k(D)$ with zero mean value. Namely, we claim that~\eqref{4.55} holds if and only if 
\begin{equation} \label{4.56}
\int_{\Gamma_i}\langle w,\nnn\rangle\dd\sigma=0\quad\mbox{for $1\le i\le m$},
\end{equation}
and in this case $p$ linearly depends on~$w$ and satisfies the inequality
\begin{equation} \label{4.056}
\|p\|_{W^{k+1,q}}\le C_1\|w\|_{W^{k,q}},
\end{equation}
where we denote by $C_i>0$ some constants depending only on~$k$ and~$q$. Indeed, suppose that~\eqref{4.55} holds and  denote by~$\chi_i\in C^\infty(\widetilde D)$ a function equal to~$1$ in a small neighbourhood of~$\overline D_i$ and to~$0$ outside a larger neighbourhood having no intersection with~$\p D\setminus \p D_i$. Then $\diver(\nabla^\bot(\chi_ip))=0$ in~$D$. Taking the integral of this relation over~$D$ and integrating by parts, we arrive at~\eqref{4.56}. Conversely, suppose that~\eqref{4.56} is fulfilled. It follows from the Leray decomposition (see Theorem~1.5 in~\cite[Chapter~I]{temam1979}) that the function $(w_2,-w_1)\in L^2(D,\R^2)$ can be written as~$\nabla p$ for some $p\in \dot H^1(D)$ if and only if 
\begin{equation} \label{4.57}
\int_D(w_2\varphi_1-w_1\varphi_2)\dd x=0\quad
\mbox{for any $\varphi=(\varphi_1,\varphi_2)\in \VV$},
\end{equation}
where $\VV$ stands for the space of infinitely smooth divergence-free vector fields on~$\R^2$ with compact support in~$D$. Thus, we need to establish~\eqref{4.57}. To this end, we define $\psi\in C^\infty(\R^2)$ by the relation
$$
\psi(x)=\int_{\gamma(a,x)}(\varphi_2\dd x_1-\varphi_1\dd x_2),
$$
where $a\in\R^2\setminus\widetilde D$ is a fixed point, and $\gamma(a,x)$ is an arbitrary smooth curve without self-intersection  going from~$a$ to~$x$. The Stokes theorem implies that~$\psi$ is a well-defined, infinitely smooth function with compact support in~$\widetilde D$ such that $\nabla^\bot\psi=\varphi$. It follows that 
\begin{align*}
\int_D(w_2\varphi_1-w_1\varphi_2)\dd x
&=- \int_D\langle w,\nabla\rangle\psi\,\dd x
=\sum_{i=1}^m\int_{\Gamma_i}\psi\langle w,\nnn\rangle\dd\sigma=0,
\end{align*}
where the last equality follows from~\eqref{4.56} and the fact that~$\psi$ is constant on each of the curves~$\Gamma_i$. We have thus shown that~$w$ can be written in the form~\eqref{4.55}. The proof of the Leray decomposition given in~\cite[Chapter~I]{temam1979}, together with the regularity theory for the boundary value problems for the Laplace operator, imply  that~$p$ is a linear function of~$w$ that satisfies~\eqref{4.056}.  

\smallskip
We now construct the operator~$\LL$. Let $z\in H^1(\DD)$ be the unique solution of the problem
$$
\Delta z=0\quad\mbox{in $\DD$}, \qquad 
\frac{\p z}{\p\nnn}\biggr|_{\p D_i}=\langle v,\nnn\rangle\bigr|_{\p D_i}
\quad\mbox{for $1\le i\le m$}, 
\qquad z\bigr|_{\p\widetilde \DD}=0.
$$
It is straightforward to check that~$z$ is a linear function of~$v$, and standard estimates for solutions of elliptic boundary value problems imply that 
\begin{equation} \label{4.58}
\|z\|_{W^{r+1,q}}\le C_2\|v\|_{W^{r,q}}\quad\mbox{for any $r\ge0$, $q\ge2$}.
\end{equation}
Let us consider the function $w=v-\nabla z$ defined in~$D$. It is not difficult to see that $w\in L_\sigma^2(D,\R^2)\cap W^{r,q}(D,\R^2)$ as soon as $v\in  W^{r,q}(D,\R^2)$ and that~$w$ satisfies~\eqref{4.56}. Thus, we can represent~$w$ in the form~\eqref{4.55}, where~$p\in \dot H^1(D)$ is a linear function of~$w$ satisfying~\eqref{4.056}. Recalling that~$z$ is also a linear function of~$v$, we see that~$p$ linearly depends on~$v$. Furthermore, it follows from~\eqref{4.056} and~\eqref{4.58} that~$p$ satisfies the inequalities
\begin{equation} \label{4.59}
\|p\|_{W^{r+1,q}}\le C_3\|v\|_{W^{r,q}}\quad\mbox{for $r\ge0$}. 
\end{equation}
Let $\LL_0:H^1(D)\to H^1(\DD)$ be an extension operator such that, for any $r\ge0$ and $q\ge2$, we have 
\begin{equation} \label{4.60}
\|\LL_0h\|_{W^{1+r,q}}\le C_4\|h\|_{W^{1+r,q}}\quad
\mbox{for $h\in W^{1+r,q}(D)$}; 
\end{equation}
see Theorem~5.22 and Remark~5.23 in~\cite{adams1975}. We now set
$$
\LL v=\nabla^\bot(\LL_0p)+\nabla z. 
$$
Then $\LL v\in L_\sigma^2(\DD,\R^2)$ and $\LL v|_D=v$. Moreover, it follows from~\eqref{4.58}--\eqref{4.60} that~\eqref{4.53} is valid. This completes the proof of the proposition.
\end{proof}

\subsection{Approximation by regular functions}
\label{s3.6}
Recall that, given a domain $D\subset\R^2$, we denote by~$V$ the space of divergence-free vector functions $v\in H^1(D,\R^2)$ vanishing on the boundary~$\p D$. We shall sometimes write $V(D)$ to indicate the domain on which the space~$V$ is considered. 

\begin{proposition} \label{p3.12}
	Let $D\subset D'\subset \R^2$ be some domains satisfying the hypotheses of Section~\ref{s3.0} (see Figure~\ref{pic2}) and let $V_D(D')$ be the subspace in~$V(D')$ that consists of the functions vanishing on $D'\setminus D$. Then there is a family of bounded linear operators $\{\Omega_\gamma:V_D(D')\to (H^2\cap V)(D')\}_{\gamma\in(0,1)}$  such that the image of~$\Omega_\gamma$ is contained in~$V_D(D')$ and 
	\begin{equation} \label{omegav}
		\|v-\Omega_\gamma v\|\le \gamma \|v\|_1
		\quad\mbox{for any $v\in V_D(D')$, $\gamma\in(0,1)$}.
	\end{equation}
\end{proposition}

\begin{proof}
As was established in the proof of Proposition~\ref{p4.9}, for any $v\in V(D')$ there is $p\in H^2(D')$ such that $v=\nabla^\bot p$ and $\|p\|_2\le C\|v\|_1$. Moreover, the construction of~$p$ given there implies that if $v\in V_D(D')$, then~$p$ can be extended to a function in~$H^2(\R^2)$ that vanishes outside~$D$. Using a partition of unity and convolution with an approximation of identity, one can construct a family of regularising operators $\omega_\gamma:H^2(\R^2)\to H^3(\R^2)$ such that $\omega_\gamma p$ vanishes on~$D^c$ if so does~$p$, and $\|\omega_\gamma p-p\|_1\le C^{-1}\gamma\,\|p\|_2$. The required family of operators can be defined by $\Omega_\gamma=\nabla ^\bot(\omega_\gamma p)$. The details of the procedure described above are very standard (e.g., see Chapter~5 in~\cite{adams1975}) and are omitted. 
\end{proof}

\addcontentsline{toc}{section}{References}
\def\cprime{$'$} \def\cprime{$'$}
  \def\polhk#1{\setbox0=\hbox{#1}{\ooalign{\hidewidth
  \lower1.5ex\hbox{`}\hidewidth\crcr\unhbox0}}}
  \def\polhk#1{\setbox0=\hbox{#1}{\ooalign{\hidewidth
  \lower1.5ex\hbox{`}\hidewidth\crcr\unhbox0}}}
  \def\polhk#1{\setbox0=\hbox{#1}{\ooalign{\hidewidth
  \lower1.5ex\hbox{`}\hidewidth\crcr\unhbox0}}} \def\cprime{$'$}
  \def\polhk#1{\setbox0=\hbox{#1}{\ooalign{\hidewidth
  \lower1.5ex\hbox{`}\hidewidth\crcr\unhbox0}}} \def\cprime{$'$}
  \def\cprime{$'$} \def\cprime{$'$} \def\cprime{$'$}
\providecommand{\bysame}{\leavevmode\hbox to3em{\hrulefill}\thinspace}
\providecommand{\MR}{\relax\ifhmode\unskip\space\fi MR }
\providecommand{\MRhref}[2]{%
  \href{http://www.ams.org/mathscinet-getitem?mr=#1}{#2}
}
\providecommand{\href}[2]{#2}

\end{document}